\documentclass[11pt]{article}

\title{\vskip-1.0em\sc Extension of derivations, and Connes-amenability of the enveloping dual Banach algebra}
\author{\sc Yemon Choi, Ebrahim Samei, Ross Stokke}
\date{17th June 2014}

\newcommand{\breakinscand}{} 
\input{WAPmac_pers}

\input{WAPmac8}

\begin{document}
\maketitle

\begin{abstract}
If $D:A \to X$ is a derivation from a Banach algebra to a contractive, Banach $A$-bimodule, then one can equip $X^{**}$ with an $A^{**}$-bimodule structure, such that the second transpose $D^{**}: A^{**} \to X^{**}$ is again a derivation. We prove an analogous extension result, where $A^{**}$ is replaced by $\F(A)$, the \emph{enveloping dual Banach algebra} of $A$, and $X^{**}$ by an appropriate kind of universal, enveloping, normal dual bimodule of~$X$.

Using this, we obtain some new characterizations of Connes-amenability of $\F(A)$. In~particular we show that $\F(A)$ is Connes-amenable if and only if $A$ admits a so-called $\operatorname{WAP}$-virtual diagonal.
We show that when $A=L^1(G)$, existence of a $\operatorname{WAP}$-virtual diagonal is equivalent to the existence of a virtual diagonal in the usual sense. Our approach does not involve invariant means for~$G$.

(primary),
43A20
43A60
46H25
(secondary).
\end{abstract}

\tableofcontents


\begin{section}{Introduction}
Amenability for Banach algebras, as introduced and studied in the pioneering work of B. E. Johnson~\cite{BEJ_CIBA}, has proved to be an important and fertile notion. However, it was recognized very early on in the development of the subject that it might not be the ``right'' notion when dealing with, say, von Neumann algebras. There is a natural variant of amenability~\cite{Connes_JFA78} that is better adapted to categories of normal bimodules over von Neumann algebras, and it turns out
that this version of amenability is equivalent to injectivity for von Neumann algebras.\footnotemark
\footnotetext{See \cite{Connes_JFA78} for an indirect proof in the case of separable preduals, which relies on results from \cite{Connes_AnnMath76}. In general, ``amenability'' implies injectivity was shown in \cite{BunPas_PAMS78}, while the results of \cite[\S\S6--7]{Connes_AnnMath76} and \cite{Elliott_AFD2} show that injective von Neumann algebras are AFD. It had been already been observed in \cite{JKR72_coho} that AFD von Neumann algebras are ``amenable'' in an appropriate sense, see Corollary 6.4 and the proof of Theorem 6.1 in that paper.}

Von Neumann algebras are particular examples of so-called \dt{dual Banach algebras}, which are roughly speaking those Banach algebras which possess a \wstar-topology that is suitably compatible with the multiplication.
For general dual Banach algebras, taking our lead from the results obtained for von Neumann algebras, one can define an analogously modified notion of amenability.
This property has become known as \dt{Connes-amenability}, and is usually defined in terms of the behaviour of certain \wstar-continuous derivations from a dual Banach algebra into \emph{normal} dual bimodules.
 Perhaps the first systematic treatment was given by V. Runde in his papers~\cite{Run_DBA1,Run_DBA2}.
 
In certain cases, Connes-amenability of natural dual Banach algebras built from a locally compact group characterizes amenability of the group, in analogy with Johnson's result that amenability of the Banach algebra $L^1(G)$ characterizes amenability of $G$. In particular, we have the following results:

\begin{thm}[Runde; Johnson]
\label{t:BEJ_VR}
Let $G$ be a locally compact group. The following are equivalent:
\begin{YCnum}
\item\label{li:G} $G$ is amenable;
\item\label{li:LG} $L^1(G)$ is amenable as a Banach algebra;
\item\label{li:WAPG*} $\WAPG^*$ is Connes-amenable as a dual Banach algebra;
\item\label{li:MG} $M(G)$ is Connes-amenable as a dual Banach algebra.
\end{YCnum}
\end{thm}

Here $M(G)$ is the \dt{measure algebra} of $G$, and $\WAPG$ denotes the space of \dt{weakly almost periodic functions on~$G$}, whose dual is equipped with a natural convolution algebra structure\footnotemark\ in the sense of \cite[Definition~19.3]{HewRoss1_ed2}.
\footnotetext{$\WAPG^*$ may also be identified with the convolution algebra of Radon measures on a certain semitopological semigroup, the WAP-compactification of~$G$, although we will not use this perspective. See \cite[\S IV.2]{BerJunMil_book} for further details.}
The implication \ref{li:G}$\implies$\ref{li:LG} is due to Johnson~\cite[Theorem 2.5]{BEJ_CIBA}; the implications \ref{li:LG}$\implies$\ref{li:WAPG*}$\implies$\ref{li:MG} are special cases of straightforward, general results on amenability and Connes-amenability for (dual) Banach algebras; and the (hard!) implication~\ref{li:MG}$\implies$\ref{li:G} is the main result of Runde's article~\cite{Run_MG1}.
We mention for sake of completeness that the implication \ref{li:LG}$\implies$\ref{li:G} is first recorded in \cite{BEJ_CIBA}, where it is attributed to J. R. Ringrose.

One of the original goals of the present work was a more direct proof of the implication \ref{li:WAPG*}$\implies$\ref{li:LG} in Theorem~\ref{t:BEJ_VR} which does not pass through amenability of $G$, and hence might shed light on similar results for more general Banach algebras.
This led to the following question: given a continuous derivation from $\LG$ to a $\LG$-bimodule $X$, is there some way to extend it to a \wswscts\ derivation from $\WAPG^*$ into some suitable dual bimodule?
We shall show that this can be done in a very natural manner. More generally, we prove
(Theorem~\ref{t:extend-der}) that one can always extend continuous derivations from a given Banach algebra $A$ to \wswscts\ derivations out of a dual Banach algebra canonically associated to~$A$.

Our approach to Theorem~\ref{t:extend-der} is motivated by work of F.~Gourdeau. In~\cite{Gou}, he proved that if $D$ is a continuous derivation from a Banach algebra $A$ to a Banach $A$-bimodule $X$, and we equip $A^{**}$ with the first Arens product, then we can equip $X^{**}$ with the structure of a Banach $A^{**}$-bimodule such that $D^{**}:A^{**}\to X^{**}$ becomes a \wswscts\ derivation. The argument has two main ingredients: first, the correspondence between derivations $A\to X$ and {\hm}s from $A$ into a certain `triangular' Banach algebra built from $A$ and $X$; second, the observation that if $\theta:A\to B$ is a continuous \hm\ between Banach algebras, then $\theta^{**}:A^{**}\to B^{**}$ is a \hm\ when both $A^{**}$ and $B^{**}$ are equipped with their first Arens products.
In our setting, we replace the second dual of a Banach algebra (which is in general not a dual Banach algebra) with a certain quotient algebra of this second dual, which can be regarded as the ``enveloping dual Banach algebra'' of a given Banach algebra.
The rest is then very similar to Gourdeau's argument for the second dual.

Once we have proved our extension result, we obtain some characterizations of Connes-amenability of the enveloping dual Banach algebra, analogous to the standard characterizations of amenability in terms of virtual diagonals. In particular, using our results on extension of derivations, we show that the enveloping dual Banach algebra of $A$ is Connes-amenable if and only if $A$ has a so-called \dt{\WAP-virtual diagonal.} Note that \cite{Run_DBA2} already characterizes  Connes-amenability of a given dual Banach algebra $B$ in terms of diagonal-type elements associated to $B$; the point of our approach is to work in a space more closely related to $A$ itself.
In the case $A=\LG$, we show that a \WAP-virtual diagonal for $\LG$ can always be ``lifted'' to a virtual diagonal for $\LG$, and we also give a description of certain submodules of $\LIGG$ that arise naturally in our approach and might be of independent interest.

Here is a summary overview of the paper.
Sections~\ref{s:prelim} and~\ref{s:triangular} set up some terminology, and review the necessary background on dual Banach algebras and triangular Banach algebras. These are then combined in Section~\ref{s:free-ndmod} to obtain the desired extension theorem for derivations (Theorem~\ref{t:extend-der}). En route, we are led to a definition of the natural ``enveloping normal dual bimodule'' (Theorem~\ref{t:FT=TF}) of a given Banach bimodule. Although our approach does not require familiarity with category-theoretic machinery, it was guided by the philosophy of adjoint functors between suitable categories, and we shall make some further comments along these lines in the relevant sections.
Section~\ref{s:functorial} collects some basic functorial properties of these constructions.

Section~\ref{s:CA-of-FA} introduces the notion of a \WAP-virtual diagonal for a given Banach algebra~$A$, and uses it to characterize Connes-amenability of the corresponding enveloping dual Banach algebra. Section~\ref{s:diag-subsp} contains technical results concerning certain subspaces of $(\AA)^*$, with respect to which one can define analogues of virtual diagonals. These are used in Section~\ref{s:anabasis} to show that if $L^1(G)$ has a $\WAP$-virtual diagonal, then it has a virtual diagonal.  The proof is a lifting argument, based on techniques used by Runde in~\cite{Run_MG1}.
In Section~\ref{s:ess-and-WAP-of-LIGG} we describe the ``essential part'' and ``WAP part'' of the $\LG$-bimodule $\LIGG$, since both these subspaces play an important role in Section~\ref{s:anabasis}. Finally, Section~\ref{s:closing} offers some closing remarks and questions.
\end{section}

\begin{section}{Preliminaries}\label{s:prelim}

\subsection{General terminology and background}
We refer to standard sources such as \cite{Bon-Dun} for the definitions of Banach algebras and Banach bimodules. In particular, our (bi)modules need not be contractive unless this is explicitly stated. However, in order to reduce needless repetition, we will adopt some terminology {\bf throughout this paper} that the reader should take heed of.
We speak only of modules or bimodules over a given Banach algebra: it should be understood that these always refer to \emph{Banach} modules or \emph{Banach} bimodules, in the sense of~\cite{Bon-Dun}. Given such a module, whenever we refer to a submodule, we always mean a \emph{closed submodule}; similarly for sub-bimodules. Throughout, all morphisms, derivations etc.~are linear and norm-continuous.

 Our definitions of \dt{dual Banach algebras} and \dt{normal dual bimodules} over them are taken from \cite{Run_DBA1}, although that paper uses  the terminology ``$\wstar$-bimodule''. For convenience and consistency of terminology, we repeat the definitions.

\begin{dfn}
Let $\fB$ be a Banach algebra, and regard $\fB^*$ as a $\fB$-bimodule in the usual way. We say $\fB$ is a \dt{dual Banach algebra with predual $\fB_*$}\/, if there exists a sub-bimodule $\fB_*\subseteq \fB^*$ such that the composition of the two natural maps $\fB\to \fB^{**}$ (embedding in the bidual) and $\fB^{**}\to (\fB_*)^*$ (adjoint of the inclusion $\fB_* \hookrightarrow \fB^*)$ is bijective. More succinctly but less precisely, this  means we require $\fB$ to itself be a dual $\fB$-bimodule.
\end{dfn}

The definition of a dual Banach algebra requires us to specify the predual, but usually we will omit this for sake of brevity, when it is clear from context what the intended predual should be.
If we wish to emphasize a particular choice of predual, we shall say ``$\fB=(\fB_*)^*$ is a dual Banach algebra''.

\begin{dfn}
Let $\fB=(\fB_*)^*$ be a dual Banach algebra, let $\fM_*$ be a $\fB$-bimodule, and let $\fM=(\fM_*)^*$ be the resulting dual $\fB$-bimodule. We say that $\fM$ is a \dt{normal dual $\fB$-bimodule} if, for each $x\in E$, the \dt{orbit maps}
\[ R^{\fB}_x: \fB\to \fM, b\mapsto b\cdot x \quad\text{and}\quad L^{\fB}_x:\fB\to \fM, b\mapsto x\cdot b \]
are both \wswscts\ (with respect to $\fB_*$ and $\fM_*$\/).
\end{dfn}

\begin{rem}\label{r:DBA-w*cts-product}
In the later paper \cite{Run_DBA2}, the following alternative definition is used for a dual Banach algebra: it is a Banach algebra $\fB$, equipped with a (not necessarily isometric) Banach space predual $\fB_*$, such that the multiplication map $\fB\times\fB\to\fB$ is \emph{separately} $\sigma(\fB,\fB_*)$-continuous. The equivalence of this with the definition given above is noted in \cite[\S1.1]{Run_DBA2}, and is a straightforward exercise which we omit.
\end{rem}

For technical reasons, we work mainly with bimodules and normal dual bimodules that are contractive. Since notions such as amenability and Connes-amenability are usually defined in terms of wider classes of bimodules, we should explain briefly why the contractive classes are good enough for our purposes.
Firstly: given a Banach algebra $A$ and an $A$-bimodule $X$, a standard renorming argument produces a contractive $A$-bimodule $X^1$ and an isomorphism of $A$-bimodules $X\iso X^1$.
Secondly: given a dual Banach algebra $\fB$ and two dual $\fB$-bimodules $\fM$ and $\fN$ which are \wstar-isomorphic as bimodules, normality of one implies normality of the other. (Indeed, if this statement were not true, then somehow normality would not be a `natural' notion for dual bimodules over a dual Banach algebra.)

\begin{lem}\label{l:renorm-ndmod}
Let $\fB$ be a dual Banach algebra and let $\fM$ be a normal dual $\fB$-bimodule. Then there exists a contractive, normal dual $\fB$-bimodule $\fN$ which is \wstar-module isomorphic to~$\fM$.
\end{lem}

\begin{proof}
By the observations before the lemma, $\fM_*^1$ is a contractive $\fB$-bimodule which is isomorphic to $\fM_*$. Hence $\fN\defeq (\fM_*^1)^*$ is a contractive dual $\fB$-bimodule that is \wstar-bimodule-isomorphic to $\fM$; and it is moreover normal, since $\fM$~is.
\end{proof}

\subsection{The enveloping dual Banach algebra}
\label{ss:free-DBA}
It is natural to consider the category $\pcat{DBA}$ of dual Banach algebras and \wswscts\ algebra {\hm}s.
At present, this category is less well understood than the usual category $\pcat{BA}$ of Banach algebras and norm-continuous {\hm}s; but the two categories are related via the existence of a ``universal enveloping dual Banach algebra'' associated to a given Banach algebra. This was shown in \cite{Run_DBA2}, but since we shall use slightly different notation in this article, we briefly review the relevant definitions.

\begin{dfn}[{\cite[Definition 4.1]{Run_DBA2}}]
Let $A$ be a Banach algebra and $E$ an 
$A$-bimodule. We denote by $\AWAPA(E)$ the set of all elements $x\in E$ for which the orbit maps $ R^A_x: A \to E, a\mapsto a\cdot x$ and $L^A_x: A \to E: a \mapsto x \cdot a$
are both weakly compact. It follows from standard properties of weakly compact linear maps -- in particular, the fact that they form an operator ideal -- that $\AWAPA(E)$ is a closed sub-bimodule of~$E$.
\end{dfn}

\begin{eg}[The motivating ur-example]
\label{eg:WAPG}
Let $A=\LG$, and regard $A^*=L^\infty(G)$ as an $A$-bimodule in the usual way. Then by a result of \"Ulger~\cite{Ulg_WAPG-cts},  $\AWAPA(A^*)$ coincides with the space $\WAPG$ of \dt{weakly almost periodic functions} on~$G$. (We recall, for sake of completeness, that $f\in \CB(G)$ is said to be weakly almost periodic if the set of left translates and the set of right translates of $f$ are relatively weakly compact subsets of $\CB(G)$.)
\end{eg}

\begin{dfn}
Given a Banach algebra $A$ and a $A$-bimodule $X$, we write $\FAX_*$ for the $A$-bimodule $\AWAPA(X^*)$, where $X^*$ is equipped with the usual bimodule action induced from~$X$.
We then define $\FAX$ to be the dual $A$-bimodule $(\FAX_*)^*$.
In the special case where $X=A$, regarded as an $A$-bimodule in the canonical way, we shall usually omit the subscripts, and simply use the notation $\FA$.

We denote by $\eta_X: X\to \FAX$ the map obtained by composing the canonical inclusion of $X$ in its second dual with the adjoint of the inclusion map $\AWAPA(X^*)\hookrightarrow X^*$\/.
Observe that $\eta_X$ is a norm-continuous $A$-bimodule map, as it is the composition of two such maps. 
\end{dfn}

The space $\FA_*=\AWAPA(A^*)$ is an example of an \dt{introverted subspace} of $A^*$: for the definition in the context of Banach algebras,
see~\cite[\S1]{LauLoy_JFA97}, although the analogous notion in the case $A=\LG$ goes back much further.
It follows from the introversion property that $\FA$ can be equipped with the structure of a Banach algebra with the following property: 
if $A^{**}$ is equipped with its first Arens product, then the adjoint of the inclusion map $\FA_*\hookrightarrow A^*$ is a \wswscts\ quotient \hm\ of Banach algebras $A^{**}\to\FA$.
(See e.g.~\cite[\S1]{LauLoy_JFA97} for details.)
 Note that this condition uniquely determines the multiplication on $\FA$, and that $\eta_A: A \to \FA$ is a \emph{norm-continuous} \hm\ with \wstar-dense range.
\begin{notn}
It is customary to denote the first Arens product in $A^{**}$ by~$\Arp$, and 
we shall use the same symbol to denote the product in $\FA$. It is a standard result (see e.g.~\cite[Theorem 3.14]{DalLau_Beurling} or~\cite[Theorem 1.4.11]{Palmer1}) that $A$ is Arens regular if and only if $\FA_*=A^*$, in which case $\FA=A^{**}$.
\end{notn}

 Runde observed (see the proof of \cite[Theorem 4.10]{Run_DBA2}) that $\FA=(\FA_*)^*$ is actually a \emph{dual} Banach algebra.
Thus, although in general $A^{**}$ is not a dual Banach algebra, it has a canonical quotient algebra which \emph{is} a dual Banach algebra, namely $\FA$.
It was also shown in~\cite{Run_DBA2} that $\FA$ is not just a canonical dual Banach algebra associated to $A$; it is a \emph{universal} one, in the following sense.

\begin{thm}[{\cite[Theorem~4.10]{Run_DBA2}}]\label{t:free-DBA}
Let $A$ be a Banach algebra, $\fB$ a dual Banach algebra, and $f:A \to \fB$ a continuous algebra \hm. Then there exists a unique \wswscts\ linear map $h:\FA\to \fB$ such that $h\eta_A=f$. Moreover, $h$ is an algebra \hm.
\end{thm}

Although this is not observed explicitly in \cite{Run_DBA2}, $\F(\blank)$ defines a functor from $\pcat{BA}$ to $\pcat{DBA}$. We could show this directly -- see Remark~\ref{r:hands-on-algebra-functor} below -- but it can also be deduced from Theorem~\ref{t:free-DBA} by ``soft'' means, as follows.

\begin{cor}\label{c:free-functor}
Let $f:A \to B$ be a continuous algebra \hm\ between Banach algebras. Then there exists a unique \wswscts\ linear map $\F(f):\FA\to\F(B)$ making the diagram
\begin{equation}
\begin{diagram}[tight,height=2em]
  \FA	& \rTo^{\F(f)}	& \F(B)  \\
 \uTo^{\eta_A}	&		& \uTo_{\eta_B}   \\
    A		&    \rTo_f	& B  
\end{diagram}
\end{equation}
commute. Moreover, $\F(f)$ is an algebra~\hm.

The assignment $f\mapsto \F(f)$ is functorial. That is: $\F(\sid_A)=\sid_{\FA}$; and if $g:B\to C$ is another continuous algebra \hm\ between Banach algebras, then $\F(gf)=\F(g)\F(f)$.
\end{cor}

\begin{proof}
The existence, uniqueness and \hm\ properties of $\F(f)$ follow by applying Theorem~\ref{t:free-DBA} to the \hm\ $\eta_B f: A\to \F(B)$.
Because of uniqueness, $\F(\sid_A)=\sid_{\FA}$.

If $f:A \to B$ and $g:B \to C$ are continuous \hm s between Banach algebras, then by the first part of this corollary, both squares in the following diagram commute:
\[ \begin{diagram}[tight,height=2em]
  \FA & \rTo^{\F(f)} & \F(B) & \rTo^{\F(g)} & \F(C) \\
  \uTo^{\eta_A} &  & \uTo^{\eta_B} &  & \uTo_{\eta_C} \\
  A & \rTo_f & B & \rTo_g & C 
\end{diagram} \]
Therefore the outer rectangle commutes, which by the uniqueness property of $\F(gf)$ in the commuting diagram
\[
\begin{diagram}[tight,height=2em,width=4em]
  \FA & \rTo^{\F(gf)} & \F(C)  \\
  \uTo^{\eta_A} &  & \uTo_{\eta_C}   \\
  A & \rTo_{gf} & C  
\end{diagram}
\]
implies that $\F(gf)=\F(g)\F(f)$.
\end{proof}

The proof of Corollary~\ref{c:free-functor} illustrates a general category-theoretic procedure. In the language of adjunctions between categories, $\F:\pcat{BA}\to\pcat{DBA}$ is \dt{left adjoint} to the forgetful functor $\pcat{DBA}\to\pcat{BA}$. A basic result in category theory tells us that to find left adjoint functors, it suffices to construct universal objects as done in Theorem~\ref{t:free-DBA}; functoriality then follows \emph{automatically} from the universal property, by a general version of the arguments used in proving Corollary~\ref{c:free-functor}. 

\begin{rem}\label{r:hands-on-algebra-functor}
Given a norm-continuous ~\hm\ $f: A \to B$,
one can define $\F(f)$ and prove its functorial nature more directly, without the machinery of adjunctions. Standard properties of weakly compact maps imply that $f^*(\F(B)_*)\subseteq \FA_*$\/, and then the adjoint of $f^*\vert_{\F(B)_*}$ is our desired map $\F(f)$. Checking that $\F(f)$ is a \hm, and that $\F(g f) = \F(g)\F(f)$, is slightly tedious but routine.
\end{rem}

\smallskip
Returning to the definition of $\FAX$ when $X$ is an arbitrary $A$-bimodule: it is important for our intended applications that $\FAX$ is not just a dual $A$-bimodule, but is in fact a \emph{normal dual $\FA$-bimodule}.
Instead of giving a direct proof, we prefer to obtain this result via the technique of triangular Banach algebras, as discussed in the next section.
\end{section}

\begin{section}{Triangular Banach algebras, modules, and derivations}
\label{s:triangular}
The following construction, which is a natural Banach-algebraic analogue of classical ideas in ring theory, appears to have been reinvented independently on several occasions.
Let $A$ be a Banach algebra and $X$ a \emph{contractive} $A$-bimodule.
We write $A\oplus_1 X$ for the $\ell^1$-sum of $A$ and $X$, and define an associative binary product on $A\oplus_1 X$ by
\begin{equation}\label{eq:define-product}
(a,x)\cdot(b,y) \defeq (ab, a\cdot y+x\cdot b) \qquad(a,b\in A, x,y\in X).
\end{equation}
Equipped with this norm and this product, $A\oplus_1 X$ becomes a Banach algebra, which we denote by~$\TAX$ and call the \dt{triangular Banach algebra} associated to~$(A,X)$. Our notation is chosen in analogy with the corresponding semidirect product construction for groups, and our terminology is chosen since one can interpret this product as given by multiplication of certain upper-triangular matrices, viz.
\[ \tmat{a}{x}\tmat{b}{y}= \tmat{ab}{a\cdot y+x\cdot b} \]
Other authors have referred to $\TAX$ as a ``module extension Banach algebra''; in older work \cite{BEJ68_wedderburn} it is called a ``strongly decomposable, topological extension of $A$ by $X$''.

\begin{rem}\label{r:why-contractive}
We have restricted to contractive bimodules to ensure that the norm on $\TAX$ is submultiplicative, so that we get a ``genuine'' Banach algebra -- this is not mentioned explicitly in \cite{Gou}, but is also needed there. If we wish to make the same construction for bimodules that are not contractive, there seem to be two reasonable options.
One could equip the Banach space $A\oplus_1 X$ with an equivalent norm for which multiplication \emph{is} submultiplicative, but the choice would be far from canonical. Alternatively one could work throughout not with Banach algebras per se, but with Banach spaces equipped with an associative product that is separately (hence jointly) norm-continuous.
\end{rem}

The second option mentioned in Remark~\ref{r:why-contractive} is arguably the more natural one, since when dealing with dual Banach algebras and weak compactness, what matters is not the norm but the topology. However, this alternative definition\footnotemark\ of a ``Banach algebra'' is less standard in the literature, so we would face difficulties in quoting results we need, and would have to repeatedly make trivial adaptations.
\footnotetext{The first author would like to thank N. Weaver (personal communication) for pointing out that \emph{this notion} was originally taken as the definition of a ``normed ring'' in the seminal work of I.~M.~Gel'fand.}
 Therefore we do not pursue this here, and have preferred to work with contractive bimodules whenever we can, just to keep the exposition cleaner. Note that a dual Banach space, equipped with an associative product that is separately \wstar-continuous, can be renormed to give a genuine dual Banach algebra in the usual sense; see, for instance, \cite[Proposition~2.1]{Daws_DBA2}.

The algebra $\TAX$ comes with two canonical maps $\imath: A\to \TAX$ and $q: \TAX\to A$, defined by
\begin{equation}\label{eq:q-and-i}
\imath(a)=(a,0) \quad\text{and}\quad q(a,x)=a\,.
\end{equation}
Clearly both maps are algebra {\hm}s.
Also, given a linear map $D:A \to X$, we define $\theta_D: A\to \TAX$ by
\begin{equation}\label{eq:section-derivation}
\theta_D(a) = (a, D(a)) \qquad(a\in A).
\end{equation}

\begin{lem}\label{l:derivation-via-HM}
Let $D:A\to X$ be a linear map. Then $D$ is a derivation if and only if $\theta_D:A \to \TAX$ is an algebra \hm. Moreover, each \hm\ $\theta: A\to \TAX$ satisfying $q\theta=\sid$ is of the form $\theta_D$ for some derivation $D:A\to X$.
\end{lem}

This correspondence between derivations and splittings of square-zero singular extensions goes back to the original work of G.~Hochschild on his eponymous cohomology groups, if not further. In the context of Banach algebras it seems to be folklore.
The proof is a simple calculation which we omit. 

\begin{rem}[A wider categorical perspective]\label{r:cat-slogan}
As $\ker(q)$ is an ideal in $\TAX$, it is naturally a contractive $\TAX$-bimodule, and hence is a contractive $A$-bimodule via the \hm\ $\imath: A\to \TAX$. In this way we may identify $X$ and $\ker(q)$ as $A$-bimodules.
Moreover, $A$-bimodule maps $X\to Y$ correspond naturally to algebra {\hm}s $\TAX\to \TR{A}{Y}$ which fix the embedded copy of $A$, while derivations $A\to X$ correspond naturally to {\hm}s $A\to \TAX$ which split the quotient map~$q$.
All this gives rise to a slogan which we wish to highlight:
\begin{quote}
One can recover the notion of ``module'' from an appropriate notion of ``split extension of algebras''; the notion of ``module morphism'' from an appropriate notion of morphism between such extensions; and the notion of ``derivation'' from an appropriate notion of splitting for such an extension.
\end{quote}
This slogan could be made much more precise using additional category-theoretic language (a readable exposition can be found in \cite[\S6.1]{Barr_AMbook}), but we will avoid this to keep the present account more focused. Nevertheless, the slogan strongly suggests that if we take the category \pcat{DBA} seriously, then we should be led naturally to consider a corresponding class of modules, module morphisms, and derivations. Moreover, the functor $\F: \pcat{BA}\to \pcat{DBA}$ should relate modules, module morphisms and derivations for a Banach algebra $A$ to corresponding notions for the dual Banach algebra $\FA$. (This line of thought is what originally led us to Proposition~\ref{p:normalmod-via-DBA} and Theorem~\ref{t:FT=TF}; the definition of $\FAX$ came later.)
\end{rem}

\medskip
We turn now to triangular Banach algebras associated to dual Banach algebras.
Given a dual Banach algebra $\fB=(\fB_*)^*$ and a dual $\fB$-bimodule $M=(M_*)^*$, there is an obvious choice of predual for $\TR{\fB}{M}$ at the level of Banach spaces, namely
\begin{equation}\label{eq:predual-of-triang}
(\TR{\fB}{M})_*\defeq \fB_*\oplus_\infty M_* \subseteq \fB^*\oplus_\infty M^*
 = (\TR{\fB}{M})^*\,.
\end{equation}
In general, however, \eqref{eq:predual-of-triang} will not be enough to make $\TR{\fB}{M}$ into a dual Banach algebra. The correct result is given in the following proposition, which should be compared with the ``slogan'' given in Remark~\ref{r:cat-slogan}. It provides further evidence that, when working with modules over a dual Banach algebra, the class of normal dual bimodules is the right one to consider.

\begin{prop}\label{p:normalmod-via-DBA}
Let $\fB=(\fB_*)^*$ be a dual Banach algebra, and let $M$ be a $\fB$-bimodule which is the dual of some Banach space $M_*$. Then the following are equivalent:
\begin{YCnum}
\item\label{li:normal-corner}
 $M$ is a normal dual $\fB$-bimodule, with predual $M_*$\/;
\item\label{li:triangular-DBA}
 $\TR{\fB}{M}$ is a dual Banach algebra, with predual $(\TR{\fB}{M})_*$\/.
\end{YCnum}
\end{prop}

After the present work was done, we found that the implication \ref{li:normal-corner}$\implies$\ref{li:triangular-DBA} has been obtained independently in \cite{bloodycomedians}, where the authors chose to omit the proof as being obvious. For sake of completeness we will give a full proof of the proposition, taking an algebraic approach rather than an argument with approximating nets.

\begin{proof}
Suppose $\TR{\fB}{M}=((\TR{\fB}{M})_*)^*$ is a dual Banach algebra, i.e. that $\fB_*\oplus_\infty M_* = (\TR{\fB}{M})_*$ is a $\TR{\fB}{M}$-submodule of $\fB^*\oplus_\infty M^* = (\TR{\fB}{M})^*$.
Let $\psi\in M_*$\/, $a\in \fB$. Then $(0,\psi)\cdot(a,0)\in \fB^*\oplus M^*$\/; by assumption it lies in $\fB_*\oplus M_*$\., and a quick calculation shows it annihilates all elements of the form $(b,0)$ where $b\in \fB$, so lies in $\{0\}\oplus M_*$\/. Since
\[ \pair{\psi\cdot a}{y} = \psi(a\cdot y) = \pair{(0,\psi)\cdot(a,0)}{(0,y)} 
 \qquad(y\in M),
\]
we see that $\psi\cdot a \in M_*$. A similar argument, with left and right reversed, shows that $a\cdot\psi\in M_*$. Thus $M_*$ is a sub-$\fB$-bimodule of $M^*$, making $M=(M_*)^*$ a dual $\fB$-bimodule.

It remains to prove $M$ is normal. Fix $y\in M$, and define $R^\fB_y: \fB \to M$ by $R^\fB_y(a)=a\cdot y$.
Given $\psi\in M_*$, we know $(0,\psi)\in \fB_*\oplus M_*$, so $(0,y)\cdot(0,\psi)\in \fB_*\oplus M_*$\/. Moreover, since
\[ \pair{(0,y)\cdot(0,\psi)}{(0,x)} = 0 \quad\text{for all $x\in M$} \]
this implies that $(0,y)\cdot(0,\psi)= (\phi,0)$ for some $\phi\in \fB_*$\/.
As
\[ \phi(a) = \pair{(0,y)\cdot(0,\psi)}{(a, 0)} = \psi(a\cdot y) = \pair{(R^\fB_y)^*(\psi)}{a}, \] 
this shows that $(R^\fB_y)^*(\psi)\in \fB_*$. Hence $(R^\fB_y)^*(M_*)\subseteq \fB_*$, and so $R^\fB_y$ is \wswscts.
A similar argument shows that the map $a\mapsto y\cdot a$ is \wswscts\ from $\fB\to M$.
Thus $M$ is normal.

Conversely, suppose $M$ is a normal dual bimodule, and fix $(\phi,\psi)\in \fB_*\oplus M_*$.
Given $(b,y)\in \fB\oplus M$, we have
\[ \begin{aligned}
\pair{(b,y)\cdot(\phi,\psi)}{(a,x)}
  = \phi(ab)+ \psi(a\cdot y)+\psi(x\cdot b) 
  = \pair{(b\cdot\phi+ (R^\fB_y)^*(\psi), b\cdot\psi)}{(a,x)}
\end{aligned} \]
for every $(a,x)\in \fB\oplus M$, where $R^\fB_y:\fB\to M$ is as defined above.
Since $\fB_*$ is a left $\fB$-submodule of $\fB^*$, we have $b\cdot\phi \in \fB_*$; since $M_*$ is a left $\fB$-submodule of $M^*$, we have $b\cdot\psi\in M_*$; and since $M$ is a \emph{normal dual} (left) $\fB$-module, we have $(R^\fB_y)^*(\psi)\in \fB_*$.
This shows that $(b,y)\cdot(\phi,\psi)\in \fB_*\oplus M_*$, so that $(\TR{\fB}{M})_*$ is a left $\TR{\fB}{M}$-submodule of $(\TR{\fB}{M})^*$.
By an exactly similar but left-right-reversed argument, one can show that
$(\phi,\psi)\cdot(b,y) \in \fB_*\oplus M_*$\/, so that $(\TR{\fB}{M})_*$ is a right $\TR{\fB}{M}$-submodule of $(\TR{\fB}{M})^*$.
Thus $\TR{\fB}{M}$ is a dual Banach algebra with predual $(\TR{\fB}{M})_*$\/.
\end{proof}

\begin{rem}
In proving this theorem, it might have been more natural (in the informal sense) to use the characterization of dual Banach algebras in terms of separately \wstar-continuous multiplication. We chose to use the original description in terms of predual modules, since the resulting argument is more in keeping with our algebraic approach, and suggests possible generalizations to other situations.
On the other hand, in the next section, particularly in the proof of Theorem~\ref{t:FT=TF}, it will be convenient to have both descriptions in mind.
\end{rem}

\end{section}

\begin{section}{Creating contractive, normal dual $\FA$-modules}
\label{s:free-ndmod}
Given a Banach algebra $A$ and an $A$-bimodule $X$, we are seeking some kind of universal, normal dual, $\FA$-bimodule associated to $X$. Our candidate is $\FAX$, but it is not immediately clear that this has all the required properties.
For instance, although it is easy to check that $\FAX_*$ is a sub $A$-bimodule of $X^*$\/, it is not immediate that it becomes an $\FA$-bimodule.
However, if $X$ is \emph{contractive} we can form the Banach algebra $\TAX$ and then by Theorem~\ref{t:free-DBA} we can form the universal enveloping dual Banach algebra $\F(\TAX)$. Our strategy is to recover $\FAX$ as the `corner' of this algebra, and then appeal to Proposition~\ref{p:normalmod-via-DBA}.
As mentioned in the introduction, this approach is similar to, and motivated by, that of Gourdeau to putting an $A^{**}$-bimodule structure on $X^{**}$, see~\cite{Gou}.

Let $V$ and $W$ be Banach spaces and let $\beta: V\times W\to\Cplx$ be a bilinear form. It follows from Gantmacher's theorem (or the Grothendieck double limit criterion) that the map $V\to W^*$, $x\mapsto \beta(x,\cdot)$, is weakly compact if and only if the map $W\to V^*$, $y\mapsto \beta(\cdot,y)$ is weakly compact. If this is the case, we shall say that $\beta$ is a \dt{weakly compact bilinear form} (this is also the terminology used in e.g.~\cite{Ulg_PAMS87}).
If $\beta:V\times W\to \Cplx$ is a weakly compact bilinear form, $X$ and $Y$ are Banach spaces, and $S:X\to V$, $T:Y\to W$ are bounded linear, then by standard properties of weakly compact operators the bilinear form $\beta\circ(S\times T): X\times Y \to \Cplx$ is also weakly compact.

If $N$ is an $A$-bimodule and $\psi\in N^*$, we define $L_\psi: A\times N\to\Cplx$ and  $R_\psi: N\times A\to\Cplx$ by 
\[
 L_\psi(a,x) = \pair{\psi}{a\cdot x} \;,\;
 R_\psi(x,a) = \pair{\psi}{x\cdot a}
\]
The previous remarks now imply the following result.

\begin{lem}\label{l:wap-bilin}
Let $M=(M_*)^*$ be a dual $A$-bimodule, and let $\psi\in M$. Then $\psi\in\AWAPA(M)$ if and only if $L_\psi$ and $R_\psi$ are weakly compact bilinear forms on $A\times M_*$ and $M_*\times A$ respectively.
\end{lem}

\begin{prop}\label{p:free-of-triang-as-dual-space}
Let $X$ be a \emph{contractive} $A$-bimodule.
If we identify $(\TAX)^*$ with $A^*\oplus_\infty X^*$, the subspace $\F(\TAX)_*$ is identified with $\FA_*\oplus_\infty \FAX_*$. 
Consequently, $\F(\TAX)$ can be identified as a dual Banach space with $\FA\oplus_1 \FAX$.
\end{prop}

\begin{proof}
For this proof, let $T$ denote $\TAX$.
Let $\pi: T \times T \to T$ be the product map.
By Lemma~\ref{l:wap-bilin}, it is enough to identify those $\Psi\in T^*$ such that $\Psi\pi$ is a weakly compact bilinear form on~$T$.
Write $\Psi=(\phi,\psi)$ where $\phi\in A^*$ and $\psi\in X^*$. Let $a_1,a_2\in A$ and $x_1,x_2\in X$; then
\begin{equation}\label{eq:WAP}
\begin{aligned}
\Psi\pi((a_1,x_1), (a_2,x_2))
 & = \Psi( a_1a_2, a_1\cdot x_2 + x_1\cdot a_2 ) \\
 & = \phi(a_1a_2) + \psi(a_1\cdot x_2) + \psi(x_1\cdot a_2) \\
 & = \phi(a_1a_2) + L_\psi(a_1, x_2) + R_\psi(a_2,x_1) .
 \end{aligned}
\end{equation}

It is clear from \eqref{eq:WAP} that if $\phi\in \AWAPA(A^*)$ and $\psi\in\AWAPA(X^*)$ then the bilinear form $\Psi\pi$ is weakly compact.
Conversely, suppose that $\Psi\pi$ is a weakly compact bilinear form. Then so is its restriction to $A\times X\subseteq T\times T$, and therefore the bilinear form $L_\psi$ is weakly compact; similarly, the bilinear form $R_\psi$ is weakly compact. Hence by Lemma~\ref{l:wap-bilin}, $\psi\in \AWAPA(X^*)=\FAX_*\/$. The restriction of $\Psi\pi$ to $A\times A$ must also be a weakly compact bilinear form, and therefore by \eqref{eq:WAP} the bilinear form $(a_1,a_2)\mapsto\phi(a_1a_2)$ is weakly compact, so that $\phi\in\AWAPA(A^*)=\FA_*$\/.
\end{proof}

We now come to the main theorems of this section. At this point it is convenient to introduce the following short-hand notation, which will also be used later in the paper. If $X$ is an $A$-bimodule, and $\eta_X : X \to \FAX$ is the canonical map (which need not be injective), then we write $\bx$ instead of $\eta_X(x)$. This just makes various formulas or chains of equations more legible.

\begin{thm}[Creation of normal bimodules]\label{t:FT=TF}
Let $X$ be a \emph{contractive} $A$-bimodule. Given $\tilx\in\FAX$ and $\tila\in \FA$, there exist unique elements $\tila\cdot\tilx$ and $\tilx\cdot\tila$ in $\FAX$ which satisfy
\[ (\tila,0)\Arp(0,\tilx) = (0,\tila\cdot\tilx)
\text{ and }
 (0,\tilx)\Arp(\tila,0) = (0,\tilx\cdot\tila) \quad\text{in $\F(\TAX)$.} \]
The operations $(\tila,\tilx)\mapsto \tila\cdot\tilx$ and
$(\tilx,\tila)\mapsto \tilx\cdot\tila$ make $\FAX$ into a $\FA$-bimodule.
Moreover
\begin{itemize}
\item we have $\ba\cdot \bx = \bre{a\cdot x}$ and $\bx\cdot\ba=\bre{x\cdot a}$ for all $a\in A$, $x\in X$;
\item $\FAX$ is a normal dual $\FA$-bimodule, with predual $\FAX_*=\AWAPA(X^*)$;
\item $\F(\TAX) = \TR{\FA}{\FAX}$ as dual Banach algebras.
\end{itemize}
\end{thm}

We note that in the conclusion of this theorem, not only is $\FAX$ made into an $\FA$-bimodule, but so is its predual $\FAX_*$.

\begin{proof}
As in the previous proposition, we write $T$ for $\TAX$. By that proposition, we can identify $\FT$ as a dual Banach space with $\FA\oplus_1\FAX$ (with predual identified with $\FA_*\oplus_\infty \FAX_*$\/).
 Observe that for $a \in A$ and $x \in X$,
$\bre{(a,x)}= (\ba, \bx).$ The natural embedding and projection maps
\[
E_{\FA} : \FA \hookrightarrow \FT,  \;    P_{\FA} : \FT \to \FA, \; P_{\FAX} : \FT \to \FAX
\]
are \wstar-continuous and satisfy $E_{\FA}(\ba) = \bre{(a, 0)}$, $P_{\FA} (\bre{(a,x)}) = \ba$.
As $\eta_A$ and $\eta_T$ are {\hm}s with \wstar-dense ranges,
 it follows from separate \wstar-\wstar\ continuity of multiplication in $\FA$ and $\FT$ (cf.~Remark~\ref{r:DBA-w*cts-product}) that $E_{\FA}$ and $P_{\FA}$ are algebra {\hm}s.
Hence, the subspace $\{(0, \tilx): \tilx \in \FAX\}= \ker P_{\FA}$ is a \wstar-closed ideal in $\FT$,
so that for any $\tila \in \FA$ and $\tilx \in \FAX$, there are elements $\tila \cdot\tilx, \ \tilx \cdot\tila \in \FAX$ such that
\[ (\tila, 0) \Arp (0, \tilx) = ( 0, \tila \cdot\tilx) \in \ker P_{\FA} \; \text{and} \; \ (0, \tilx) \Arp (\tila, 0) = ( 0, \tilx\cdot\tila ) \in \ker P_{\FA}\,, \]
where $\Arp$ denotes the Arens product in $\FT$.
In this way, $\FAX$ is a Banach $\FA$-bimodule: associativity and the bimodule property follow from associativity of multiplication in~$\FT$. 

Now let $\tila, \tilb \in \FA$, $\tilx, \tily \in \FAX$. Observe that
\[ \{(\tila, 0): \tila \in \FA\}= \operatorname{Im} E_{\FA} =  \ker P_{\FAX}\]
is a \wstar-closed subalgebra of $\FT$, isometrically isomorphic as a dual Banach algebra with~$\FA$.
Hence
\[ (\tila, 0 ) \Arp (\tilb, 0) = E_{\FA} (\tila) \Arp E_{\FA}(\tilb) = E_{\FA} (\tila \Arp \tilb) = (\tila \Arp \tilb, 0).\]
Taking nets $(x_i), (y_j)$ in $X$  such that $\bx_i \to \tilx$, $\by_j \to \tily$ \wstar in $\FAX$, we obtain
\[ \begin{aligned}
(0,\tilx) \Arp (0,\tily)
 & = \wslim_i \wslim_j (0, \bx_i) \Arp (0, \by_j) \\
 & = \wslim_i \wslim_j \bre{(0, x_i)(0,y_j)} & = (0,0). \\
\end{aligned} \]

\noindent
Hence, $(\tila, \tilx) \Arp (\tilb, \tily) = (\tila \Arp \tilb, \tila \cdot\tily + \tilx \cdot \tilb)$,
which is exactly the product in the triangular algebra $\TR{\FA}{ \FAX}$. Thus, we can identify the dual Banach algebra $\FT$ with $\TR{\FA}{ \FAX}$.
By Proposition~\ref{p:normalmod-via-DBA}, $\FAX$ is a normal dual $\FA$-bimodule. (Note that this proposition ensures $\FAX$ is not merely a module which happens to be a dual space, but that it is a dual module.)
\end{proof}

Here is the promised ``extension'' theorem for derivations, which should be compared with \cite[Lemma~2.2]{Gou}.

\begin{thm}[`Extension' of derivations]
\label{t:extend-der}
Let $A$ be a Banach algebra, $X$ a \emph{contractive} $A$-bimodule, and $D:A \to X$ a continuous derivation. Then  there exists a unique \wswscts\ derivation $\til{D}: \FA\to \FAX$ which makes the diagram
\[ \begin{diagram}[tight,height=2.5em]
 \FA & \rTo^{\til{D}} & \FAX \\
 \uTo^{\eta_A} & & \uTo_{\eta_X} \\
 A & \rTo_{D} & X
\end{diagram} \]
commute.
\end{thm}

\begin{proof}
Let $q:\TAX\to A$ be the canonical quotient \hm\ and $\imath: A\to \TAX$ the canonical inclusion \hm.
Define $\theta_D:A\to\TAX$ by $\theta_D(a)=(a,D(a))$; this is a norm-continuous algebra \hm, by 
 Lemma~\ref{l:derivation-via-HM}.
Then $\F(\theta_D): \FA\to \F(\TAX)$ is a \wswscts\ algebra \hm\ between dual Banach algebras (Corollary~\ref{c:free-functor}).

By Theorem~\ref{t:FT=TF} we may identify $\F(\TAX)$ with $\TR{\FA}{\FAX}$, whereupon $\F(q)$ is just the canonical quotient \hm\ $\TR{\FA}{\FAX}\to \FA$.
Since $q\theta_D=\sid$, we have $\F(q)\F(\theta_D)=\sid$ by functoriality, and so by Lemma~\ref{l:derivation-via-HM}, $\F(\theta_D)=\theta_{\til{D}}$ for some derivation $\til{D}: \FA\to \FAX$. As $\F(\theta_D)\eta_A = \eta_{(\TAX)} \theta_D$,
it follows that $\til{D}\eta_A=\eta_X D$.

Since $\theta_{\til{D}}(\tila) = (\tila, \til{D}(\tila))$ for all $\tila\in\FA$, the \wstar-\wstar continuity of $\theta_{\til{D}}=\F(\theta_D)$ implies that $\til{D}$ is \wswscts. Uniqueness follows either by using uniqueness of the extended \hm\ $\F(\theta_D)$, or more directly by noting that $\eta_A(A)$ is \wstar-dense in $\FA$.
\end{proof}

\begin{rem}\label{r:non-contractive}
We stated the theorem only for contractive $A$-bimodules, because our proof goes through triangular Banach algebras (see Remark~\ref{r:why-contractive}). This restriction could easily be dropped if we were prepared to think of Banach algebras as Banach spaces equipped with continuous associative multiplication, and relaxed the condition that the norm be submultiplicative. We would then find that the proofs of Theorems \ref{t:FT=TF} and~\ref{t:extend-der} go through without changes, and so both theorems \emph{are valid} even for bimodules that are not contractive. (One could also employ a slightly ad hoc renorming argument, see the comments before Lemma~\ref{l:renorm-ndmod}.) However, contractive bimodules are good enough for our purposes in the present paper.
\end{rem}

\begin{rem}
Given Theorem~\ref{t:extend-der}, if $\eta_A$ and $\eta_X$ are injective with norm-closed range and $\til{D}$ is inner, one might hope to find a net of inner derivations $A\to X$ which approximate~$D$. When $\AWAPA(X^*)=X^*$ then this can be done using Goldstine's lemma and Mazur's theorem, just as in Gourdeau's article \cite{Gou}. However, it is not clear if this will work in general.
\end{rem}

\end{section}

\begin{section}{Some functorial properties of $\F_A$, and a canonical map from $\FA\times\FA$ to $\F_A(\AA)$}
\label{s:functorial}
Our next main goal is to use Theorem~\ref{t:extend-der} to give an alternative description of when the dual Banach algebra $\FA$ is Connes-amenable. This will be done in Section~\ref{s:CA-of-FA}, but we need to prepare for this by setting up some general machinery. Moreover, the approach in Section~\ref{s:CA-of-FA} uses certain elements in the bimodule $\F_A(\AA)$ which behave in some sense like virtual diagonals for~$A$. Handling such elements requires us to take a closer look at how $\FA\ptp\FA$ and $\F_A(\AA)$ are related.

Note that any $\FA$-bimodule may be regarded as an $A$-bimodule via the \hm\ $\eta_A:A\to\FA$, and we shall do this automatically throughout this section and the following ones.

\begin{thm}[Universal property of $\FAX$]\label{t:unit-for-module-adjunction}
Let $X$ be a contractive $A$-bimodule and $\fN$ a contractive, normal dual $\FA$-bimodule. Let $g: X\to \fN$ be an $A$-bimodule map; then there exists a unique
\wswscts\ $\FA$-bimodule map $\gwap:\FAX\to \fN$ such that $\gwap\eta_X=g$.
\end{thm}

\begin{proof}
Since $\eta_X(X)$ is \wstar-dense in $\FAX$, there can be at most one \wswscts\ map $\gwap$ with the required properties.

Let $\psi\in N_*$. By \cite[Proposition~4.2]{Run_DBA2} (or a direct argument using normality of $\fN$), the $\FA$-orbit maps of $\psi$ are weakly compact as maps $\FA\to N_*$. Hence, when we regard $\fN$ as an $A$-bimodule, the $A$-orbit maps $R^A_\psi$ and $L^A_\psi$ are weakly compact.
Consider $g^*(\psi)\in X^*$; since $g^*$ is an $A$-bimodule map, the $A$-orbit maps of $g^*(\psi)$ factorize as
$R^A_{g^*(\psi)} = g^*\circ R^A_\psi$ and $L^A_{g^*(\psi)} = g^*\circ L^A_\psi$\/.
Since $R^A_\psi$ and $L^A_\psi$ are weakly compact,
 so are $R^A_{g^*(\psi)}$ and $L^A_{g^*(\psi)}$, i.e.\ $g^*(\psi)\in \AWAPA(X^*)$.

Thus $g^*(\fN_*)\subseteq \AWAPA(X^*)=\FAX_*$. Define $\gwap$ to be the adjoint of $g^*: \fN_* \to \FAX_*$. By construction $\gwap$ is an $A$-bimodule map, and since $\FAX$ and $\fN$ are both normal dual modules, a straightforward \wstar-approximation argument shows $\gwap$ is in fact a $\FA$-bimodule map.
\end{proof}

\begin{cor}[Naturality of $X\to\FAX$]
\label{c:natural}
Let $X$ and $Y$ be contractive $A$-bimodules, and let $f:X\to Y$ be an  $A$-bimodule map.
Then there exists a unique \wswscts\ $\FA$-module map $\F_A(f):\FAX\to\FAY$ such that $\F_A(f)\eta_X=\eta_Yf$.
\end{cor}

\begin{proof}
Apply Theorem~\ref{t:unit-for-module-adjunction} with $\fN=\FAY$ and $g=\eta_Yf$.
\end{proof}

\begin{cor}[Co-unit property for $\F_A$]\label{c:counit-for-module-adjunction}
Let $\fM=(\fM_*)^*$ be a contractive, normal dual $\FA$-bimodule; regard it as an $A$-bimodule in the natural way, and consider the normal dual $\FA$-bimodule $\F_A(\fM)$. Then there is a \wswscts\ $\FA$-bimodule map $\veps_\fM: \F_A(\fM)\to \fM$ satisfying $\veps_\fM\eta_\fM=\sid$.
\end{cor}

\begin{proof}
Apply Theorem~\ref{t:unit-for-module-adjunction} with $X=\fM$, $\fN=\fM$ and $g=\text{identity map on~$\fM$}$.
\end{proof}

\begin{rem}
We can describe the \hm\ $\veps_\fM:\F_A(\fM)\to \fM$ from Corollary~\ref{c:counit-for-module-adjunction} more explicitly, as follows. The main calculation used in the proof of Theorem~\ref{t:unit-for-module-adjunction} shows that
$\fM_* \subseteq \AWAPA(\fM^*)$, and the adjoint of the inclusion map $\fM_*\hookrightarrow \AWAPA(\fM^*)$ is then the desired quotient \hm~$\veps_\fM$.
\end{rem}

\begin{rem}[Adjoint functors, slight return]
Recall that Theorem~\ref{t:free-DBA} and the ensuing corollaries can be encapsulated in the statement
\begin{quote}
``$\F$~is left adjoint to the forgetful functor from dual Banach algebras to Banach algebras''.
\end{quote}
Likewise, Theorem~\ref{t:unit-for-module-adjunction} and the ensuing corollaries can be encapsulated in the statement
\begin{quote}
``$\F_A$~is left adjoint to the forgetful functor from contractive normal dual $\FA$-bimodules to contractive $A$-bimodules.''
\end{quote}
\end{rem}


\begin{rem}
One can deduce Theorem~\ref{t:unit-for-module-adjunction} from earlier results in this paper, without explicit use of \cite[Proposition 4.2]{Run_DBA2}. The idea is as follows: consider the map $f:\TAX\to \TR{\FA}{ \fN}$, $(a,x) \mapsto (\ba, g(x))$; this is a norm-continuous algebra \hm, and $\TR{\FA}{\fN}$ is a dual Banach algebra (Proposition~\ref{p:normalmod-via-DBA}), hence by Runde's universality result (Theorem~\ref{t:free-DBA}) there is a unique \wswscts\ algebra \hm\ $h:\F(\TAX)\to \TR{\FA}{\fN}$ such that $f=h\eta_{\TAX}$. Identifying $\F(\TAX)$ with $\TR{\FA}{\FAX}$ and restricting $h$ to the embedded copy of $\FAX$, we obtain the desired map $\gwap:\FAX\to\fN$ which ``extends'' $g$. 

This approach would be in keeping with our theme of deducing results about modules from results about algebras, and would reinforce the philosophy that all constructions are dictated by naturality and the properties of the functor $\F$ on the category of Banach algebras. However, there are two shortcomings. Firstly, spelling out all the algebraic details is rather tedious, and leads to a proof that seems overly indirect. Secondly, this approach relies on Theorem~\ref{t:free-DBA}, which is itself proved using~\cite[Pro\-pos\-ition~4.2]{Run_DBA2}; so we might as well use that proposition directly.
\end{rem}

We now turn to aspects of the bimodule $\F_A(\AA)$.
Recall that the Arens products on $A^{**}$ can be constructed \footnotemark\ by first defining two bilinear maps $A^{**}\times A^{**}\to (\AA)^{**}$ that extend the canonical bilinear map $A\times A \to \AA$; composing each of these maps with $\Delta^{**}: (\AA)^{**}\to A^{**}$ then yields the left and right Arens products.
\footnotetext{This appears to be an old observation, perhaps even known to Arens himself. We do not know the first explicit reference, but a discussion of this point of view can be found in \cite[\S3]{Daws_Diss10}.}
Following a suggestion of the referee, we can approach the Arens-type product on $\FA=\AWAPA(A^*)$ in the same way, factorizing it through a suitable bilinear map $\Theta:\FA\times\FA\to \F_A(\AA)$.

We start in greater generality.
Let $M$ and $N$ be left and right $A$-modules, respectively, and let $\Theta_0: M\times N \to M\ptp N$ be the canonical bilinear map.
(Left) Arens extension of $\Theta_0$ proceeds in three stages. First we define $\Theta_1: (M\ptp N)^* \times M \to N^*$ as follows: if $\beta\in (M\ptp N)^*$ and $m\in M^*$, let
\begin{equation}\label{eq:define_Theta_1}
\Theta_1(\beta, m) : n \mapsto \beta(m \tp n) \qquad(n\in N).
\end{equation}
Next, we define  $\Theta_2: N^{**}\times (M\ptp N)^* \to M^*$ as follows: for $\tiln\in N^{**}$ and $\beta\in (M\ptp N)^*$,
\begin{equation}\label{eq:define_Theta_2}
\Theta_2(\tiln, \beta): m \to \pair{\tiln}{\Theta_1(\beta,m)} \qquad(m\in M).
\end{equation}
Finally, we define $\Theta_3: M^{**}\times N^{**} \to (M\ptp N)^{**}$, the \dt{(left) Arens extension of $\Theta_0$}, as follows: for $\tilm\in M^{**}$ and $\tiln\in N^{**}$,
\begin{equation}\label{eq:define_Theta_3}
\Theta_3(\tilm,\tiln) : \beta \mapsto \pair{\tilm}{\Theta_2(\tiln,\beta)}
\quad (\beta\in (M\ptp N)^*).
\end{equation}
Note that $\Theta_3(\blank,\tiln)$ is \wswscts\ for each $\tiln\in N^{**}$. 

Regarding $M\ptp N$ as an $A$-bimodule in the standard way, we equip $(M\ptp N)^*$ with the dual bimodule structure.
The next lemma is a routine if somewhat tedious calculation.

\begin{lem}\label{l:extension-is-A-bimodule-map}
Let $a\in A$. Then:
\begin{YCnum}
\item\label{li:PEOPLE ON}
$\Theta_1(a\cdot\beta,m)=a\cdot\Theta_1(\beta,m)$
and
$\Theta_1(\beta\cdot a,m)=\Theta_1(\beta, a\cdot m)$
for all $\beta\in (M\ptp N)^*$, $m\in M$;
\item\label{li:THIS TRAIN}
$\Theta_2(\tiln,a\cdot\beta)=\Theta_2(\tiln\cdot a,\beta)$
and
$\Theta_2(\tiln,\beta\cdot a) = \Theta_2(\tiln,\beta)\cdot a$
for all $\tiln\in N^{**}$, $\beta\in (M\ptp N)^*$;
\item\label{li:WON'T SHUT UP}
$\Theta_3(a\cdot\tilm,\tiln) = a\cdot\Theta_3(\tilm,\tiln)$
and
$\Theta_3(\tilm,\tiln\cdot a) = \Theta_3(\tilm,\tiln)\cdot a$
for all $\tilm\in M^{**}$, $\tiln\in N^{**}$.
\end{YCnum}
\end{lem}

\begin{proof}
The proof of part \ref{li:PEOPLE ON} is a straightforward argument using the definition of $\Theta_1$ and the definition of the module action on $(M\ptp N)^*$. 
Then, using part \ref{li:PEOPLE ON} and the definition of $\Theta_2$ (Equation~\eqref{eq:define_Theta_2}), we find that for $\tiln\in N^{**}$, $\beta\in (M\ptp N)^*$ and $m\in M$ we have
\[ \begin{aligned}
 \pair{\Theta_2(\tiln, a\cdot\beta)}{m}
 & = \pair{\tiln}{\Theta_1(a\cdot\beta,m)} \\
 & = \pair{\tiln}{a\cdot\Theta_1(\beta,m)} 
 = \pair{\tiln\cdot a}{\Theta_1(\beta,m)}
 = \pair{\Theta_2(\tiln\cdot a, \beta)}{m}
\end{aligned} \]
and
\[ \begin{aligned}
 \pair{\Theta_2(\tiln, \beta\cdot a)}{m}
& = \pair{\tiln}{\Theta_1(\beta\cdot a, m)} \\
& = \pair{\tiln}{\Theta_1(\beta, a\cdot m)} 
 = \pair{\Theta_2(\tiln, \beta)}{a\cdot m}
 = \pair{\Theta_2(\tiln, \beta)\cdot a}{m}
\end{aligned} \]
which proves \ref{li:THIS TRAIN}. In a similar way, using part \ref{li:THIS TRAIN} and Equation~\eqref{eq:define_Theta_3}, we can show that
for $\tilm\in M^{**}$, $\tiln\in N^{**}$ and $\beta\in (M\ptp N)^*$
we have
\[ \pair{\Theta_3(a\cdot\tilm,\tiln)}{\beta}
 = \pair{a\cdot\Theta_3(\tilm,\tiln)}{\beta}\quad,\quad
\pair{\Theta_3(\tilm,\tiln\cdot a)}{\beta}
  = \pair{\Theta_3(\tilm,\tiln)\cdot a}{\beta} \;; \]
the details are left to the reader.
\end{proof}


Given a left $A$-module $M$, let 
$\LWAP_A(M) =\{ \mu\in M \st R^A_\mu \text{ is weakly compact}\}$,
and for a right $A$-module $N$, define $\RWAP_A(N)$ in the analogous way.

\begin{lem}\label{l:sort-of-introverted}\
\begin{YCnum}
\item\label{li:using-LWAP}
If $\beta\in \LWAP_A((M\ptp N)^*)$ then $\Theta_1(\beta,m) \in \LWAP_A(N^*)$ for all $m\in M$.
\item\label{li:using-RWAP}
If $\beta\in \RWAP_A((M\ptp N)^*)$ then $\Theta_2(\tiln,\beta)\in \RWAP_A(M^*)$ for all $\tiln\in N^{**}$.
\item\label{li:introversion}
If $\tilm\in M^{**}$ and $\tiln\in \LWAP_A(N^*)^\perp$ then $\Theta_3(\tilm,\tiln)\in \LWAP_A((M\ptp N)^*)^\perp$.
If $\tilm\in \RWAP_A(M^*)^\perp$ and $\tiln\in N^{**}$ then $\Theta_3(\tilm,\tiln)\in \RWAP_A((M\ptp N)^*)^\perp$.
\end{YCnum}
\end{lem}

\begin{proof}
Let $\beta \in (M\ptp N)^*$, $m\in M$ and $\tiln\in N^{**}$.
By Lemma~\ref{l:extension-is-A-bimodule-map}\ref{li:PEOPLE ON},
$R^A_{\Theta_1(\beta,m)} = \Theta_1( R^A_\beta(\blank), m)$ as maps $A\to N^*$\/,
and by
Lemma~\ref{l:extension-is-A-bimodule-map}\ref{li:THIS TRAIN},
$L^A_{\Theta_2(\tiln,\beta)}= \Theta_2(\tiln, L^A_\beta(\blank))$ as maps $A \to M^*$\/.

So if $\beta\in\LWAP_A((M\ptp N)^*)$ then $R^A_{\Theta_1(\beta,m)}$ factors through a weakly compact map, hence is weakly compact, which proves part~\ref{li:using-LWAP}.
On the other hand, if $\beta\in \RWAP_A((M\ptp N)^*)$ then $L^A_{\Theta_2(\tiln,\beta)}$ factors through a weakly compact map, hence is weakly compact, which proves part~\ref{li:using-RWAP}.

Finally: part \ref{li:introversion} follows easily from the results proved in parts~\ref{li:using-LWAP} and~\ref{li:introversion}, once we recall how $\Theta_3$ is defined
(see Equation~\eqref{eq:define_Theta_3}).
\end{proof}

Now we specialize to the case $M=N=A$.
It is well-known that when considering weak almost periodicity of functions on groups, it suffices to look at either left translates or right translates without having to check both. The following observation is an abstract version of this result; since we did not find an explicit statement in the literature we consulted, we include a proof for the reader's convenience.

\begin{lem}\label{l:LWAP=RWAP}
Let $\psi\in A^*$. Then 
$R^A_\psi = (L^A_\psi)^*\circ\kp$ and
$L^A_\psi = (R^A_\psi)^*\circ\kp$.
Consequently, $\LWAP_A(A^*)=\RWAP_A(A^*)=\AWAPA(A^*)$.
\end{lem}

\begin{proof}
Let $\psi\in A^*$, $a\in A$, $b\in B$. Then
\[
\pair{R^A_\psi(a)}{b} = \pair{a\cdot \psi}{b}
 = \pair{\psi\cdot b}{a} 
 = \pair{L^A_\psi(b)}{a} 
 = \pair{\kp(a)}{L^A_\psi(b)}
  = \pair{(L^A_\psi)^*\kp(a)}{b} \ .
\]
Thus $R^A_\psi = (L^A_\psi)^*\circ\kp$, and the other identity is proved similarly.
The final part of the lemma now follows from Gantmacher's theorem that an operator is weakly compact if and only if its adjoint is.
\end{proof}

Consider $\Theta_3:A^{**}\times A^{**}\to (\AA)^{**}$.
 By Lemma \ref{l:sort-of-introverted}\ref{li:introversion} and Lemma \ref{l:LWAP=RWAP}, basic linear algebra yields a bounded bilinear map
$\Theta: \FA \times \FA \to \F_A(\AA)$
which makes the following diagram commute:
\begin{equation}\label{eq:making Theta}
\begin{diagram}[tight,height=2.5em,width=4.5em]
A^{**}\times A^{**}  & \rTo^{\Theta_3} & (\AA)^{**} \\
\dTo^{{q_A\times q_A}}  & & \dTo_{q_{\AA}} \\
\FA \times \FA& \rTo_{\Theta} & \F_A(\AA)
\end{diagram}
\end{equation}
Here, $q_A: A^{**}\to \FA$ and $q_{\AA} : (\AA)^{**} \to \F_A(\AA)$ denote the natural quotient maps.
Since $\eta_A=q_A\kp_A$ and $\eta_{\AA}=q_{\AA}\kp_{\AA}$, we have
$\Theta(\eta_A\times \eta_A) = \eta_{\AA}\Theta_0$, or in our abbreviated notation,
\[ \Theta(\bv{a},\bv{b}) =\bv{a\tp b} \qquad\text{ for all $a,b\in A$.} \]

\begin{rem}\label{r:DWAP Theta is product}
By chasing through the definitions, one can show that $\Delta^{**}\Theta_3: A^{**}\times A^{**} \to A^{**}$ is precisely the usual first Arens product $(\tila,\tilb)\to \tila\Arp\tilb$.
Also, there is a commutative diagram
\begin{equation}\label{eq:making DWAP}
 \begin{diagram}[tight,height=2.5em,width=4.5em]
(\AA)^{**} & \rTo^{\Delta^{**}}  & A^{**} \\
\dTo^{q_{\AA}} & & \dTo_{q_A} \\
\F_A(\AA) & \rTo_{\DWAP} & \FA
\end{diagram}
\end{equation}
(c.f.\ the proof of Theorem~\ref{t:unit-for-module-adjunction}).
 Combining the commutative diagrams \eqref{eq:making Theta} and \eqref{eq:making DWAP}, we see that $\DWAP\Theta:\FA\times\FA\to\FA$ is just the multiplication map for the algebra $\FA$.
(If one revisits the original calculations of Lau and of Runde, this is in effect how one \emph{defines} the multiplication map on $\FA$ to make it into a dual Banach algebra.)
\end{rem}

\begin{lem}[$A$-bimodule property of $\Theta$]\label{Theta is an A-bimodule map}
For any $\tilm,\tiln\in\FA$ and $x\in A$ we have
 $\Theta(\tilm,\tiln\cdot x) = \Theta(\tilm,\tiln)\cdot x$ and 
 $\Theta(x\cdot\tilm,\tiln) = x\cdot\Theta(\tilm,\tiln)$.
 \end{lem}

\begin{proof}
Lift $\tilm,\tiln$ to representatives $\tilm', \tiln'\in A^{**}$.
By Lemma~\ref{l:extension-is-A-bimodule-map}\ref{li:WON'T SHUT UP}, we have
$\Theta_3(\tilm',\tiln'\cdot x) = \Theta_3(\tilm',\tiln')\cdot x$.
Now apply $q_{\AA}$ to both sides. Noting that $q_A$ and $q_{\AA}$ are $A$-bimodule maps, and using Diagram~\ref{eq:making Theta}, we get
\[ \Theta(\tilm,\tiln\cdot x) = q_{\AA} \Theta_3(\tilm',\tiln'\cdot x) = q_{\AA}\left[ \Theta_3(\tilm',\tiln')\cdot x\right] = \Theta(\tilm,\tiln)\cdot x \]
as required. The other identity is proved similarly and we omit the details.
\end{proof}

It is natural to ask if $\Theta$ is separately \wswscts\ as a bilinear map between dual Banach spaces. The answer turns out to be positive if $\FA$ has an identity element, but showing this requires some work. However, for the intended applications in Section~\ref{s:CA-of-FA}, we only need the following weaker version.

\begin{prop}\label{p:top centre idea}
Let $b,c\in A$. Then the linear maps
 $\tilx\mapsto \Theta(\tilx,\bv{c})$
and
 $\tilx\mapsto \Theta(\bv{b},\tilx)$
are both  \wswscts\ from $\FA\to \F_A(\AA)$.
\end{prop}

\begin{proof}
 As observed earlier, $\Theta_3$ is \wswscts\ in the first variable.
Therefore, considering Diagram~\ref{eq:making Theta} and noting that $q_A: A^{**}\to \FA$ is a quotient map for the \wstar-topologies, we deduce that $\Theta$ is \wswscts\ in the first variable.
To handle the second variable, fix $a\in A$. Then
$\pair{\Theta_3(\kappa(a),\tily)}{\beta}
 = \pair{\kappa(a)}{\Theta_2(\tily,\beta)} = \pair{\tily}{\Theta_1(\beta,a)}$ for all $\tily\in A^{**}$ and $\beta\in (\AA)^*$. This shows that $\Theta_3(\kappa(a),\,\cdot\,) : A^{**} \to (\AA)^{**}$ is \wswscts. By the definition of $\Theta$, there is a commutative diagram
\[ \begin{diagram}[tight,height=2.5em,width=5.5em]
A^{**} & \rTo^{\Theta_3(\kappa(a),\,\cdot\,)} & (\AA)^{**} \\
\dTo^{q_A}  & & \dTo_{q_{\AA}} \\
\FA & \rTo_{\Theta(\bv{a}, \,\cdot\,)} & \F_A(\AA)
\end{diagram} \]
and since the vertical arrows are quotient maps for the \wstar-topologies,
 $\Theta(\bv{a},\,\cdot\,)$ is \wswscts.
\end{proof}


For sake of completeness, in the remainder of this section we show how one can push Proposition~\ref{p:top centre idea} further, at least when $\FA$ has an identity element.

\begin{thm}[$\Theta$ is an $\FA$-bimodule map]
\label{Theta is an FA-bimodule map}
Let $\tila,\tilb,\tilc\in \FA$.
Then
$\tila\cdot\Theta(\tilb,\tilc)=\Theta(\tila\Arp\tilb,\tilc)$ and
and $\Theta(\tila,\tilb\Arp\tilc ) = \Theta(\tila,\tilb)\cdot\tilc$.
\end{thm}

\begin{cor}\label{c:THUNDERCATS HO}
Suppose $\FA$ has an identity element $\tile$. Then $\Theta:\FA\times\FA\to \F_A(\AA)$ is separately \wswscts.
\end{cor}

\begin{proof}[Proof of the corollary]
By Theorem~\ref{Theta is an FA-bimodule map}, for every $\tila,\tilb\in\FA$ we have
$\Theta(\tila,\tilb) = \tila\cdot\Theta(\tile,\tile)\cdot\tilb$.
The result now follows from normality of $\F_A(\AA)$.
\end{proof}

If Corollary~\ref{c:THUNDERCATS HO} held without assuming $\FA$ has an identity element, then Theorem \ref{Theta is an FA-bimodule map} would follow quickly from Proposition \ref{Theta is an A-bimodule map} by \wstar-\wstar\ continuity.
As pointed out by the referee, some assumption on $A$ is necessary: for if $A$ is a Banach space equipped with the zero product then $\FA=A^{**}$, $\F_A(\AA)=(\AA)^{**}$ and $\Theta=\Theta_3$ is usually not \wswscts\ in the 2nd variable. Therefore, to prove Theorem \ref{Theta is an FA-bimodule map} we take an indirect route.

\begin{lem}\label{l:halfway there}
Let $a\in A$ and $\tilb\in \FA$. Then $\Theta(\bv{a},\tilb\Arp\tilc ) = \Theta(\bv{a},\tilb)\cdot\tilc$ for all $\tilc\in \FA$.
\end{lem}

\begin{proof}
Given $a\in A$ and $\tilb\in\FA$ we define two linear maps $f_1,f_2:\FA\to\F_A(\AA)$ by
$f_1(\tilc) \defeq \Theta(\bv{a},\tilb\Arp\tilc)$, 
$f_2(\tilc) \defeq \Theta(\bv{a},\tilb)\cdot\tilc$\/.
Since $\FA$ is a dual Banach algebra, $L^{\FA}_{\tilb}: \tilc\mapsto \tilb\Arp\tilc$ is \wswscts; therefore, by Proposition~\ref{p:top centre idea}, $f_1$ is \wswscts. Since $\F_A(\AA)$ is a normal dual $\FA$-bimodule, $f_2$ is also \wswscts.
By Lemma~\ref{Theta is an A-bimodule map}, $f_1$ and $f_2$ agree on $\eta_A(A)$, which is \wstar-dense in $\FA$, and so they agree everywhere on~$\FA$.
\end{proof}

\begin{proof}[Proof of Theorem~\ref{Theta is an FA-bimodule map}]
We start by recalling that $\Theta$ is \wswscts\ in the first variable (see the proof of Proposition~\ref{p:top centre idea}).
Now, fix $\tilb,\tilc\in\FA$, and define maps $g_1, g_2: \FA \to \F_A(\AA)$ by
$g_1(\tila)\defeq \tila\cdot\Theta(\tilb,\tilc)$ and $g_2(\tila)=\Theta(\tila\Arp\tilb,\tilc)$, for $\tila\in\FA$.
Since $\F_A(\AA)$ is a \emph{normal} $\FA$-bimodule, $g_1$ is \wswscts. $g_2$ is also \wswscts, since $\Theta$ is \wswscts\ in the first variable and multiplication in a dual Banach algebra is separately \wswscts.
By Lemma \ref{Theta is an A-bimodule map}, $g_1$ and $g_2$ agree on $\eta_A(A)$, and so they agree on all of $\FA$ by \wstar-density.

The second identity requires a similar idea, but more than just a simple left-right switch.
Consider the maps $h_1,h_2: \FA \to \F_A(\AA)$ defined by
$h_1(\tila)\defeq \Theta(\tila,\tilb\Arp\tilc)$
and $h_2(\tila)\defeq\Theta(\tila,\tilb)\cdot\tilc$, for $\tila\in\FA$.
$h_1$ is \wswscts, since $\Theta$ is \wswscts\ in the first variable;
and since $\F_A(\AA)$ is a \emph{dual} $\FA$-bimodule, $h_2$ is also \wswscts. By Lemma \ref{l:halfway there}, $h_1$ and $h_2$ agree on $\eta_A(A)$, so once again by \wstar-density they coincide on~$\FA$.
\end{proof}

\end{section}


\begin{section}{Connes-amenability of $\FA$}
\label{s:CA-of-FA}
Recall that a dual Banach algebra $\fB$ is said to be \dt{Connes-amenable} if each \wswscts\ derivation from $\fB$ to a normal dual $\fB$-bimodule is inner.
In this section, we apply the algebraic machinery developed in previous sections to give an alternative description of when $\FA$ is Connes-amenable.

\begin{rem}\label{r:CA-con}
By our previous remarks on renorming normal dual bimodules, to decide if a dual Banach algebra $\fB$ is Connes-amenable, it suffices to only consider derivations into \emph{contractive} normal dual $\fB$-bimodules.
\end{rem}

We require a small observation that is not new but is worth stating explicitly.

\begin{thm}[Connes-amenability of $\FA$]
\label{t:CA-via-usual}
Let $A$ be a Banach algebra. Then the following statements are equivalent:
\begin{YCnum}
\item\label{li:CA}
 $\FA$ is Connes-amenable;
\item\label{li:der1}
 for every contractive $A$-bimodule $M$ and every derivation $D:A\to M$, the derivation $\eta_M D: A\to \F_A(M)$ is inner;
\item\label{li:der2}
 Every derivation from $A$ into a contractive, normal dual $\FA$-bimodule is inner.
\end{YCnum}
\end{thm}

\begin{proof}[Proof that \ref{li:CA}$\implies$\ref{li:der1}]
Let $M$ be a contractive $A$-bimodule and $D:A\to M$ a derivation. Applying Theorem~\ref{t:extend-der}, there exists a \wswscts\ derivation $\til{D}:\FA\to\F_A(M)$ which makes the following square commute:
\[ \begin{diagram}[tight,height=2.5em,width=4em]
 \FA & \rTo^{\til{D}} & \F_A(M) \\
 \uTo^{\eta_A} & & \uTo_{\eta_{M}} \\
 A & \rTo_{D} & M
\end{diagram} \]
As $\F_A(M)$ is normal, Connes-amenability of $\FA$ implies that $\til{D}$ is inner; hence $\til{D}\eta_A$ is inner. Since $\til{D}\eta_A= \eta_MD$, \ref{li:der1} holds.

\paragraph{\it Proof that \ref{li:der1}$\implies$\ref{li:der2}}
Let $\fM$ be a contractive, normal dual $\FA$-bimodule and $d:A\to\fM$ be a derivation. By hypothesis, the derivation $D=\eta_{\fM} d: A \to \F_A(\fM)$ is inner, say $D= {\rm ad}_{\tily}$ for some $\tily\in \F_A(\fM)$.
By Corollary~\ref{c:counit-for-module-adjunction} there is a \wswscts\ $A$-bimodule map $\veps_\fM: \F_A(\fM) \to \fM$ such that $\veps_{\fM} \eta_{\fM} = \sid$.
Hence, if we put $\tilw=\veps_{\fM}(\tily)$,
we have
$d(a) = \veps_{\fM} D(a) = \veps_{\fM} (a\cdot \tily - \tily \cdot a) = a \cdot \tilw - \tilw \cdot a = {\rm ad}_{\tilw} (a)$ for all $a\in A$.
Thus $d$ is inner.


\paragraph{\it Proof that \ref{li:der2}$\implies$\ref{li:CA}}
Let $\fN$ be a contractive, normal dual $\FA$-bimodule and let $D:\FA\to \fN$ be a \wswscts\ derivation. Then $D\eta_A:A\to\fN$ is a derivation, so by hypothesis $D\eta_A = {\rm ad}_{\tiln}:A\to\fN$ for some $\tiln\in\fN$. Since $\fN$ is a normal dual module, ${\rm ad}_{\tiln}$ extends to a \wswscts\ inner derivation ${\rm ad}_{\tiln}:\FA\to\fN$. As $D$ and ${\rm ad}_{\tiln}$ are \wswscts\ and agree on the \wstar-dense subset $\eta_A(A)$, they agree on all of $\FA$, and so $D={\rm ad}_{\tiln}$ is inner. In view of Remark~\ref{r:CA-con}, $\FA$ is Connes-amenable.
\end{proof}

The theorem has the following consequence, which we will need later.

\begin{cor}\label{c:CA-via-der}
Let $A$ be a Banach algebra. The following are equivalent:
\begin{itemize}
\item[{\rm(a)}] $\FA$ is Connes-amenable;
\item[{\rm(b)}] $\FA$ has an identity element, and every derivation from $A$ into a unit-linked, contractive, normal dual $\FA$-bimodule is inner.
\end{itemize}
\end{cor}

\begin{proof}
Suppose (a) holds. By \cite[Proposition 4.1]{Run_DBA1} $\FA$ has an identity element, $\tile$ say. The rest of~(b) follows by \ref{li:CA}$\implies$\ref{li:der2} of Theorem~\ref{t:CA-via-usual}.

Conversely, suppose (b) holds. Let $\tile$ be the identity element of $\FA$, and let $\fN$ be a contractive, normal dual $\FA$-bimodule. Let $I$ denote the formal identity operator $\fN\to\fN$; then we have an isomorphism of normal dual $\FA$-bimodules
\[ \fN = (I-\tile)\fN(I-\tile) \oplus \tile\fN(I-\tile) \oplus (I-\tile)\fN\tile \oplus \tile\fN\tile.\]
Now let $D:\FA\to\fN$ be a \wswscts\ derivation. By a standard argument for derivations on unital Banach algebras (see, for instance, the remarks in \cite[\S1.c]{BEJ_CIBA}), we can write $D=D_0+D_1$ where $D_1$ takes values in the \emph{unit-linked}, normal dual bimodule $\tile\fN\tile$, and $D_0$ is inner. By the assumption (b), $D_1$ is inner, so $D$ is inner, and (a) holds.
\end{proof}

Just as amenability for Banach algebras may be characterized in terms of the existence of virtual diagonals, Connes-amenability for dual Banach algebras may be characterized in terms of the existence of certain ``diagonal-type'' elements. To be more precise: the paper \cite{Run_DBA2} introduces the notion of a \dt{$\sigWC$-diagonal} for a given dual Banach algebra $\fB$, which is by definition a certain kind of linear functional on a particular subspace of $(\fB\ptp\fB)^*$, and proves [Theorem~4.8, \ibid] that $\fB$ is Connes-amenable if and only if $\fB$ has a $\sigWC$-diagonal.

In the case where $\fB=\FA$ for some Banach algebra $A$, we can obtain a more convenient characterization with the same flavour, phrased in terms of certain elements in $\F_A(\AA)$.
To state the relevant definition, we first need some notation.

\begin{notn}
We write $\Delta$ for the product map $\AA\to A$ and $\DWAP$ for the induced map $\F_A(\Delta):\F_A(\AA)\to\FA$, which is well-defined by Corollary~\ref{c:natural}.
\end{notn}

\begin{dfn}\label{d:WAP-VD}
An element $\tilm \in \F_A(\AA)$ is called a \dt{\WAP-virtual diagonal for $A$} if
$a \cdot \tilm = \tilm \cdot a$ and $\DWAP(\tilm) \cdot a = \bv{a}$ for each $a \in A$\/.
\end{dfn}

\begin{rem}\label{r:WAPVD-extra}
If $\tilm\in\F_A(\AA)$ is a \WAP-virtual diagonal for $A$, then it follows from normality of $\F_A(\AA)$ and \wstar-continuity that $\tila\cdot\tilm =\tilm\cdot\tila$
and $\DWAP(\tilm) \cdot \tila = \tila$
 for all $\tila\in\FA$ (and not just those in $\eta_A(A)$).
\end{rem}

Runde has shown in \cite{Run_TAMS06} that whenever $G$ is an amenable non-compact [SIN] group, $\WAP(G)^*$ is a Connes-amenable dual Banach algebra which fails to have a normal virtual diagonal.
Nevertheless, we will show below (Theorem~\ref{t:WAPVD-CA}) that $\FA$ is Connes-amenable if and only if $A$ has a \WAP-virtual diagonal.
We shall take a somewhat indirect route, first discussing and constructing a ``universal'' derivation for $\FA$ in the case where $\FA$ is unital. This follows the philosophy, going back to Hochschild's original papers, that in order to show cohomology groups with coefficients in some class of modules vanish, it is often enough to show this for one particular module in that class. The framework developed along the way may be of interest in other problems concerning derivations out of dual Banach algebras.

\begin{dfn}
Let $A$ be a Banach algebra such that $\FA$ has an identity element.
We denote by $\DERFA$ the class of all pairs $(D,\fN)$ where $\fN$ is a \emph{unit-linked, contractive}, normal dual $\FA$-bimodule and $D:A\to \fN$ is a (norm-continuous) derivation. A pair $(d,\fX)\in\DERFA$ is said to be \dt{weakly universal for $\DERFA$} if, for each $(D,\fN)\in\DERFA$, there is a \wswscts\  $A$-bimodule map $g:\fX\to\fN$ such that $g d = D$.
\end{dfn}

\begin{rem}
This terminology is motivated by category theory. Universal constructions can often be interpreted as being initial objects in certain categories, and in the present case one can make $\DERFA$ into a category in a natural way. There is a notion of a ``weakly initial object'' in a given category, and in this case it would correspond to a weakly universal element of $\DERFA$ in the sense defined above.
\end{rem}


\begin{thm}\label{t:weakly-universal}
Suppose $\FA$ has an identity element $\tile$. Then there exists $\tils\in \F_A(\AA)$ with the following properties:
\begin{itemize}
\item $\DWAP(\tils)=\tile$;
\item if we put $d_A(a) = \tils\cdot a - a\cdot\tils$ for all $a\in A$, then $d_A:A \to \ker\DWAP$ is weakly universal for $\DERFA$.
\end{itemize}
\end{thm}

The proof of Theorem~\ref{t:weakly-universal} is patterned after known ideas for derivations from unital Banach algebras. However, since $A$ need not have an identity element, we cannot obtain $d_A$ directly from a derivation $A\to\ker\Delta$ in any reasonable way, and must work a little harder.
The next lemma is inspired by the proof of \cite[Lemma 4.9]{Run_DBA2}. (Note that there is a typographical error in the statement of that lemma; the conclusion should be that a certain map takes values in $\sigWC(({\mathfrak A}\ptp{\mathfrak A})^*)$.

\begin{lem}\label{l:surprising-extension}
Let $D:A\to N$ be a (norm-continuous) derivation from $A$ into an $A$-bimodule $N$,
and define $h:\AA\to N$ by $h(b\tp c) = b\cdot D(c)$.

\begin{YCnum}
\item\label{li:basic-identity}
 For all $a,b,c\in A$ we have
$a \cdot h(b\tp c) 
=  h(ab\tp c)$ and
$h(b\tp c)\cdot a = h(b\tp ca) - bc\cdot D(a)$.

\item\label{li:surprisingly-WAP}
 $h^*(\AWAPA(N^*))\subseteq \AWAPA((\AA)^*)$.
\end{YCnum}
\end{lem}

\begin{proof}
Part \ref{li:basic-identity} is proved by direct calculation using the definition of $h$ and the derivation identity for~$D$.
To prove part \ref{li:surprisingly-WAP} it is convenient to use Lemma~\ref{l:wap-bilin}. Let $\psi \in \AWAPA(N^*)$ and consider the bilinear forms $L_{h^*(\psi)}: A\times \AA \to \Cplx$, $R_{h^*(\psi)}:\AA\times A\to\Cplx$. From part \ref{li:basic-identity} we have
\[ \begin{aligned}
L_{h^*(\psi)}(a,b\tp c)
 & = \pair{h^*(\psi)}{ab\tp c} \\
 & = \pair{\psi}{h(ab\tp c)} \\
 & = \pair{\psi}{a\cdot h(b\tp c)} = L_\psi(a, h(b\tp c)) \qquad(a,b,c\in A).
\end{aligned} \]
By linearity and continuity $L_{h^*(\psi)}= L_\psi\circ(\sid\times h)$, which is a weakly compact bilinear form since $L_\psi$~ is.

Part~\ref{li:basic-identity} also implies that
\[ \begin{aligned}
R_{h^*(\psi)}(b\tp c, a)
 & = \pair{h^*(\psi)}{b\tp ca} \\
 & = \pair{\psi}{h(b\tp ca)} \\
 & = \pair{\psi}{h(b\tp c)\cdot a} + \pair{\psi}{bc\cdot D(a)} \\
 & = R_\psi(h(b\tp c), a) + L_\psi(bc, D(a)) 
	 \qquad(a,b,c\in A).
\end{aligned} \]
By linearity and continuity, $R_{h^*(\psi)} = R_\psi\circ (h\times \sid) + L_\psi\circ (\Delta\times D)$, which is the sum of two weakly compact bilinear forms and so is weakly compact.
Thus $h^*(\psi)\in \AWAPA((\AA)^*)$, and this proves part~\ref{li:surprisingly-WAP}.
\end{proof}

The next lemma makes use of the map $\Theta:\FA\times\FA\to\F_A(\AA)$ that was 
 defined in the previous section (see Diagram~\ref{eq:making Theta}), as suggested by the referee in response to an earlier proof of Theorem~\ref{t:weakly-universal}.

\begin{lem}\label{l:hwap}
Let $h$ be as in Lemma~\ref{l:surprising-extension}, and define $\hwap: \F_A(\AA)\to \F_A(N)$ to be the adjoint of the map $h^*: \AWAPA(N^*)\to \AWAPA((\AA)^*)$.
\begin{YCnum}
\item\label{li:extending-nonmodule-map}
 $\hwap(\bv{b\tp c}) = \bv{h(b\tp c)}$ for all $b,c\in A$;
\item\label{li:restriction-of-hwap}
 the restriction of $\hwap$ to $\ker\DWAP$ is an
$A$-bimodule map from $\ker\DWAP$ to $\F_A(N)$.
\end{YCnum}
Moreover, if we let $\til{D}:\FA\to \F_A(N)$ be the extension of $D$ provided by Theorem~\ref{t:extend-der}, then
\[
\hwap\Theta(\bv{b},\tilc)=\bv{b}\cdot\til{D}(\tilc)
\quad\text{and}\quad
\hwap\Theta(\tilb,\bv{c})=\tilb\cdot\til{D}(\bv{c}) = \tilb\cdot\bv{D(c)}
\]
 for all $b,c\in A$ and all $\tilb$, $\tilc\in \FA$.
\end{lem}

\begin{proof}
Part~\ref{li:extending-nonmodule-map} is a direct calculation.
For part~\ref{li:restriction-of-hwap}, let $a\in A$ and $\tilw\in \F_A(\AA)$. Using part~\ref{li:extending-nonmodule-map} of the present lemma, part \ref{li:basic-identity} of Lemma~\ref{l:surprising-extension}, and \wstar-continuity, we obtain
\[
\begin{aligned}
a\cdot\hwap(\tilw) & = \hwap(a\cdot\tilw) \\
\hwap(\tilw)\cdot a & = \hwap(\tilw\cdot a)
- \DWAP(\tilw) \cdot \eta_N D(a).
\end{aligned}
\]
Hence, given $\tilw\in\ker\DWAP$, we have $a\cdot\hwap(\tilw) = \hwap (a\cdot\tilw)$ and $\hwap(\tilw)\cdot a = \hwap(\tilw\cdot a)$ for all $a\in A$.


For the last part of the lemma: by part~\ref{li:extending-nonmodule-map} we have
\begin{equation}\label{eq:BEER}
\bv{b}\cdot \til{D}(\bv{c}) = \bv{b}\cdot \bv{D(c)} = \bv{b\cdot D(c)} = \hwap\Theta(\bv{b},\bv{c}) \quad\text{for all $b,c\in A$.} 
\end{equation}
Fixing $b\in A$, consider the two maps $\FA\to \F_A(\AA)$ defined by $\tilc\mapsto \bv{b}\cdot\til{D}(\tilc)$ and $\tilc\mapsto \hwap\Theta(\bv{b},\tilc)$. The first map is \wswscts, since $\til{D}$ is and $\F_A(\AA)$ is a dual module;
the second map is \wswscts\ by Proposition~\ref{p:top centre idea}; and they agree on $\eta_A(A)$ by \eqref{eq:BEER}. Therefore they agree on all of~$\FA$.
Similarly, if we fix $c\in A$, we consider  the two maps $\FA\to \F_A(\AA)$ defined by $\tilb\mapsto \tilb\cdot\til{D}(\bv{c})$
 and $\tilb\mapsto \hwap\Theta(\tilb,\bv{c})$. The first map is \wswscts, since $\F_A(\AA)$ is a \emph{normal} dual module;
the second map is \wswscts, again by Proposition~\ref{p:top centre idea}; and the two maps agree on $\eta_A(A)$ by \eqref{eq:BEER}. Therefore they agree on all of~$\FA$, and this completes the proof.
\end{proof}

\begin{rem}
We would have liked to construct the map $\hwap$ as some kind of extension of an existing bimodule map. The problem is that although $h\vert_{\ker\Delta}: \ker\Delta \to N$ is an $A$-bimodule map, the extension given ``by abstract nonsense'' would be $\F_A(h): \F_A(\ker\Delta)\to \F_A(N)$, and it is not obvious how to show $\F_A(\ker\Delta)$ coincides with $\ker\DWAP$.
\end{rem}

\begin{proof}[Proof of Theorem~\ref{t:weakly-universal}]
Let $\tils\defeq \Theta(\tile,\tile)$.
Then for each $a\in A$ we have
\[ d_A(a) \defeq \tils\cdot a - a\cdot\tils = \Theta(\tile, \bv{a}) -\Theta(\bv{a},\tile), \]
the second equality following from Lemma~\ref{Theta is an A-bimodule map}. It is easily checked that $d_A:A\to\F_A(\AA)$ is a derivation, which takes values in the unit-linked, normal dual $\FA$-bimodule $\ker\DWAP$ (see Remark~\ref{r:DWAP Theta is product}).

Now let $(D,\fN)\in\DERFA$.
Define $h:Z \to \fN$ by $b\tp c\mapsto b\cdot D(c)$, and let $\hwap:\F_A(\AA)\to\F_A(\fN)$ be the \wswscts\ map $\hwap:\F_A(\AA)\to \F_A(\fN)$ produced by Lemma~\ref{l:surprising-extension}. By Lemma~\ref{l:hwap},  
$\hwap\vert_{\ker\DWAP}$ is an $A$-bimodule map, and it also satisfies
\[
\hwap\Theta(\bv{b},\tilc)=\bv{b}\cdot\til{D}(\tilc)
\quad\text{and}\quad
\hwap\Theta(\tilb,\bv{c})=\tilb\cdot\til{D}(\bv{c}) = \tilb\cdot\bv{D(c)}
\]
 for all $b,c\in A$ and all $\tilb$, $\tilc\in \FA$.
(Here $\Dtild:\FA\to\F_A(\fN)$ is the \wswscts\ ``extension'' of $D$ that is provided by Theorem~\ref{t:extend-der}.)
Since $\F_A(\fN)$ is unit-linked, a standard argument for derivations into unit-linked bimodules shows us that $\Dtild(\tile)=0$.
Therefore
\[
\hwap d_A(a)
  = \hwap\Theta(\tile,\bv{a}) -\hwap\Theta(\bv{a},\tile) 
  = \tile\cdot\til{D}(\bv{a}) - \bv{a}\cdot\til{D}(\tile)
  = \bv{D(a)} \quad\text{for all $\tila\in \FA$}.
\]
So $f\defeq\veps_{\fN} \hwap\vert_{\ker\DWAP}$ is an $A$-bimodule map satisfying $f  d_A=D$, as required.
\end{proof}


\begin{thm}\label{t:WAPVD-CA}
Let $A$ be a Banach algebra. Then $\FA$ is Connes-amenable if and only if $A$ has a $\WAP$-virtual diagonal.
\end{thm}

\begin{proof}
Either hypothesis -- Connes-amenability of $\FA$, or the existence of a $\WAP$-virtual diagonal for~$A$ -- implies that $\FA$ has an identity element, $\tile$ say. 
So by the first part of Theorem~\ref{t:weakly-universal}, there exists $\tils\in\FAZ$ such that $\DWAP(\tils)=\tile$.

Suppose $\FA$ is Connes-amenable. As $\ker\DWAP$ is a normal dual $\FA$-bimodule, and the derivation $\FA\to\ker\DWAP$, $\tila \mapsto \tils\cdot \tila - \tila\cdot \tils$ is \wswscts, this derivation must be inner. Hence there exists $\tiln\in\ker\DWAP$ with $a\cdot\tiln - \tiln\cdot a = \tils\cdot a - a\cdot \tils$ for all $a\in A$. Rearranging, we find that $\tils+\tiln$ is a $\WAP$-virtual diagonal for $A$.

Conversely, suppose $A$ has a $\WAP$-virtual diagonal~$\tilm\in\FAZ$.
 Then $\tilm-\tils\in \ker\DWAP$, and
\begin{equation}\label{eq:oldman}
a \cdot (\tilm-\tils) - (\tilm-\tils)\cdot a = - \tils\cdot a + a \cdot\tils =  \quad\text{for all $a\in A$.}
\end{equation}
Let $(D,\fN)\in\DERFA$, and let $d_A:A \to\ker\DWAP$ be the derivation $d_A(a)=\tils\cdot a -a\cdot\tils$ ($a\in A$). By Theorem~\ref{t:weakly-universal} the pair $(d_A,\ker\DWAP)$ is weakly universal for $\DERFA$, hence there exists an $A$-bimodule map $f: \ker\DWAP\to \fN$ such that $f d_A = D$. Setting $\tily = f(\tilm-\tils)\in \fN$ we deduce from \eqref{eq:oldman} that
\[ D(a) = f(a \cdot(\tilm-\tils)-(\tilm-\tils)\cdot a) = a \cdot \tily - \tily \cdot a \quad(a\in A). \]
Since $D$ is \wswscts\ and $\fN$ is normal, $D= {\rm ad}_{\tily}$. By Corollary~\ref{c:CA-via-der}, $\FA$ is Connes-amenable.
\end{proof}

\begin{rem}
As pointed out to us by the referee, work of Daws characterizes the Connes-amenability of $\FA$ by the existence of quasi-expectations for representations of $A$ on reflexive Banach spaces, see \cite[Proposition 6.15]{Daws_DBA2}.
Daws obtains this as a special case of more general results on dual Banach algebras: in one direction, a $\sigWC$-diagonal is used to construct quasi-expectations; and in the other direction (which is in our view harder), a quasi-expectation for a suitable representation is shown to give rise to a $\sigWC$-diagonal.
It seems likely that by adapting his proofs in the obvious way, one could use a \WAP-virtual diagonal to construct quasi-expectations, and obtain a \WAP-virtual diagonal from a well-chosen quasi-expectation. Since this is somewhat orthogonal to the goals of the present paper, and since we do not have anything new to add to the arguments given in \cite[\S6]{Daws_DBA2}, we will not discuss this circle of ideas.
\end{rem}

One might hope that, under suitable conditions on a Banach algebra $A$, we can lift a \WAP-virtual diagonal to obtain a virtual diagonal, thereby giving a proof (for such examples) that Connes-amenability of $\FA$ implies amenability of~$A$. The next section sets up some machinery which can assist~us.
\end{section}

\begin{section}{Diagonal-type functionals on subspaces of $(\AA)^*$}\label{s:diag-subsp}
Throughout this section, $A$ is a Banach algebra and $\Delta:\AA\to A$ is as in previous sections. We keep to our standing conventions that a bimodule over a Banach algebra always means a ``Banach bimodule'', and that a sub-bimodule is always assumed to be closed.
For convenience, we also adopt the common abbreviation ``b.a.i.'' to stand for ``bounded approximate identity''.

\begin{notn}
For several choices of closed subspace $E\subseteq A^*$, we will consider the adjoint of the map $\Delta^*: E\to (\AA)^*$. Rather than the cumbersome notation $\left(\Delta^*\vert_E\right)^*$, we shall write $\Delta_E$ for this adjoint. By abuse of notation, if $V\subseteq (\AA)^*$ is a closed subspace that contains $\Delta^*(E)$, we shall also use $\Delta_E$ to denote the adjoint of the map $\Delta^* : E \to V$; it should be clear from context which subspace $V$ is being considered. This notation, which will be used several times in Section~\ref{s:anabasis},  is in analogy with our use of $\DWAP$ to denote the adjoint of $\Delta^* : \WAP(A^*)\to \AWAPA((\AA)^*)$.
\end{notn}

\begin{dfn}\label{d:V-vd}
Let $V\subseteq (\AA)^*$ be a sub-$A$-bimodule, and let $E=(\Delta^*)^{-1}(V)\subseteq A^*$. We say $V$ is \dt{diagonally suitable} if $E$ separates points of $A$ (that is, for each non-zero $a\in A$, there exists $\psi\in E$ with $\psi(a)\neq 0$).
Given such a $V$, we say that a functional $F\in V^*$ is a \dt{$V$-virtual diagonal} for $A$ if
\begin{equation}
a\cdot F = F\cdot a \quad\text{and}\quad \pair{\Delta_E(F)\cdot a}{\phi} = \pair{\phi}{a}  \qquad(a\in A,\phi\in E).
\end{equation}
\end{dfn}

\begin{rem}\
\begin{itemize}
\item[(a)]
One could clearly make the same definition without requiring $V$ to be diagonally suitable, but then this allows trivial cases where $F=0$.
\item[(b)]
The second condition in the definition of a $V$-virtual diagonal $F\in V^*$ is rather weak.
For even if $E=(\Delta^*)^{-1}(V)$ is a norming subspace of $A^*$, and $u\in A^{**}$ satisfies
$\pair{u\cdot a}{\phi} = \pair{\phi}{a}$ for all $a\in A$ and $\phi\in E$,
this does not guarantee that $\pair{u\cdot a}{\psi}=\pair{\psi}{a}$ for all $a\in A$ and $\psi\in A^*$. 
\item[(c)]
Our notion of a $V$-virtual diagonal is distinct from the notion of a \dt{$\Phi$-virtual diagonal} that is considered in~\cite{CorGal_BA97}, which was introduced to generalize the notion of a \dt{normal virtual diagonal} for a \emph{dual} Banach algebra.
\end{itemize}
\end{rem}

\begin{eg}\
\begin{enumerate}
\item
Suppose $A$ is a non-zero Banach algebra. By the Hahn-Banach theorem, $(\AA)^*$ is itself diagonally suitable, and an $(\AA)^*$-virtual diagonal is just a virtual diagonal in the usual sense.
\item
Suppose $A$ is a non-zero Banach algebra, and let $W$ denote $\AWAPA ((\AA)^*)$. Recall that $W$ is a sub-$A$-bimodule of $(\AA)^*$, and that 
\[  \Delta^*(\FA_*)\subseteq \F_A(\AA)_* \equiv W. \]
Therefore, if $\eta_A:A\to\FA$ is injective,  $W$ is diagonally suitable.
Assume moreover that $\Delta:\AA \to A$ is surjective (this is always the case if $A$ has a b.a.i., for instance). Then by Lemma~\ref{l:patch-WAP} below,
$(\Delta^*)^{-1}(W) = \AWAPA(A^*)\equiv \FA_*$\/, and so for such $A$ a $W$-virtual diagonal for $A$ is the same thing as a \WAP-virtual diagonal in the sense of Definition~\ref{d:WAP-VD}.
\end{enumerate}
\end{eg}

\begin{lem}\label{l:patch-WAP}
Suppose that $\Delta^*:A^*\to (\AA)^*$ is bounded below (equivalently, that $\Delta: \AA \to A$ is surjective). Then
$\AWAPA(A^*) = \{ \psi \in A^* \st \Delta^*(\psi)\in \AWAPA((\AA)^*) \}$.
\end{lem}

\begin{proof}
Let $\psi\in A^*$. Since $\Delta^*:A^*\to (\AA)^*$ is an $A$-bimodule map,
\begin{equation}\label{eq:BAYLISS}
\begin{aligned}
\Delta^*( \{ a\cdot \psi \st a\in A, \norm{a}\leq 1\})
 & = \{ a\cdot\Delta^*(\psi) \st a\in A, \norm{a}\leq 1\}, \\
\Delta^*( \{ \psi\cdot a \st a\in A, \norm{a}\leq 1\})
 & = \{ \Delta^*(\psi)\cdot a \st a\in A, \norm{a}\leq 1\}.
\end{aligned}
\tag{$*$}
\end{equation}
Given two Banach spaces $X$ and $Y$ and a bounded linear map $f:X\to Y$ which is bounded below, a subset $S\subseteq X$ is relatively weakly compact if and only if $f(S)$ is. (This is a straightforward consequence of the fact that $f^*:Y^*\to X^*$ is surjective.) Lemma~\ref{l:patch-WAP} now follows from the equations \eqref{eq:BAYLISS} and the definition of weakly almost periodic elements.
\end{proof}

Recall that an $A$-bimodule $X$ is said to be \dt{neo-unital} or \dt{pseudo-unital} if each $x\in X$ can be written as $a\cdot y\cdot b$ for some $a,b\in A$ and $y\in Y$.

\begin{lem}[The essential part of a bimodule]
\label{l:ess-part}
Suppose $A$ has a b.a.i., and let $X$ be an $A$-bimodule. The subspace
$X_{\ess} \defeq \lin\{a\cdot x\cdot b \st a,b\in A, \ x\in X\}$
is a neo-unital sub-$A$-bimodule of $X$. Moreover, there exists a projection of $A$-bimodules from $X^*$ onto the sub-bimodule $X_{\ess}^\perp$.
\end{lem}

\begin{proof}
This is a standard argument using Cohen's factorization theorem, and can be found in the proof of \cite[Proposition 1.8]{BEJ_CIBA}.
\end{proof}

\begin{rem}\label{r:sub-of-neo}
Suppose $A$ has a b.a.i.~and $X$ is a neo-unital $A$-bimodule. It follows from a variant of Cohen's factorization theorem (see \cite[Theorem 11.10]{Bon-Dun}) that every $A$-submodule of $X$ is also neo-unital.
In particular, every $A$-submodule of $X$ is naturally a bimodule for the \emph{multiplier algebra} $M(A)$.
(For background on the \dt{multiplier algebra} of a Banach algebra, see e.g.~\cite[\S1.2]{Palmer1}, with the caveat that this source uses the older terminology of ``double centralizer algebra''. The construction of $M(A)$-actions on neo-unital $A$-bimodules can be found in \cite[\S1.d]{BEJ_CIBA}; a recent exposition is given in \cite[Theorem 3.2]{Daws_Diss10}.)
\end{rem}

\begin{notn}
To make some of the formulas which follow more legible, we adopt the convention that if $M$ is an $A$-bimodule then $M^*_{\ess}$ denotes the essential part of $M^*$, i.e.~we omit the parentheses. The dual of $M_{\ess}$, if we ever need it, will be denoted by $(M_{\ess})^*$.
\end{notn}

If $A$ has a b.a.i., it follows from the Hahn-Banach theorem and the definition of $(\,\cdot\,)_{\ess}$ that $A^*_{\ess}$ is a separating subset of~ $A^*$. Moreover, $\Delta^*(A^*_{\ess})\subseteq (\AA)^*_{\ess}$, as $\Delta^*$ is an $A$-bimodule map.  Thus $(\AA)^*_{\ess}$ is diagonally suitable, and we can consider the notion of an $(\AA)^*_{\ess}$-virtual diagonal. But before proceeding, we should identify $(\Delta^*)^{-1}\left[(\AA)^*_{\ess}\right]$ explicitly, in view of the second part of Definition~\ref{d:V-vd}.

\begin{lem}\label{l:patch-ess}
Suppose $A$ has a~b.a.i. Then
$A^*_{\ess} = \{ \psi \in A^* \st \Delta^*(\psi)\in (\AA)^*_{\ess} \}$.
\end{lem}

\begin{proof}
Since $\Delta^*$ is an $A$-bimodule map, $\Delta^*( A^*_{\ess}) \subseteq (\AA)^*_{\ess}$.
Conversely, let $\psi\in A^*$ be such that $\Delta^*(\psi)\in (\AA)^*_{\ess}$. Let $(e_i)$ be a b.a.i.~for $A$: then
\begin{equation}
\label{eq:PEMBLETON}
\lim_i \norm{ e_i\cdot\Delta^*(\psi) - \Delta^*(\psi)} = \lim_i \norm{\Delta^*(\psi)\cdot e_i - \Delta^*(\psi)} = 0 .
\tag{$\dagger$}
\end{equation}
Now, $e_i\cdot\Delta^*(\psi)=\Delta^*(e_i\cdot\psi)$ and $\Delta^*(\psi)\cdot e_i = \Delta^*(\psi\cdot e_i)$. Also, since $A$ has a b.a.i, $\Delta:\AA\to A$ is surjective, so $\Delta^*:A^*\to (\AA)^*$ is bounded below. Therefore, \eqref{eq:PEMBLETON} implies that 
\[ \lim_i \norm{ e_i\cdot\psi - \psi} = \lim_i \norm{\psi\cdot e_i - \psi} = 0 .\]
It follows easily that $\psi \in A^*_{\ess}$, as required.
\end{proof}

\begin{prop}\label{p:near-VD}
Let $A$ be a Banach algebra with a~b.a.i. If $A$ has an $(\AA)^*_{\ess}$-virtual diagonal, then it has a virtual diagonal, and hence is amenable.
\end{prop}

\begin{proof}
Throughout this proof, we let $V$ denote $(\AA)^*$, just to improve legibility. By the remarks before Lemma~\ref{l:patch-ess}, $V$ is diagonally suitable.
By Lemma~\ref{l:ess-part}, there is an $A$-bimodule projection $P$ from $(\AA)^{**}$ onto the sub-bimodule $V^\perp$. The natural quotient map $(\AA)^{**}\to V^*$ factors through $I-P$, yielding an isomorphism of $A$-bimodules $\imath: (I-P)(\AA)^{**} \to V^*$, which satisfies
$\imath^{-1}(T)\vert_V = T$ for all $T\in V^*$.

Now suppose $F\in V^*$ is a $V$-virtual diagonal for $A$. Let $M=\imath^{-1}(F)\in (\AA)^{**}$. Then
\begin{equation}\label{eq:boyle}
 a\cdot M = \imath^{-1}(a\cdot F) = \imath^{-1}(F\cdot a ) = M\cdot a
 \quad\text{for all $a\in A$.}
\end{equation}
Moreover, let $a\in A$ and $\psi\in A^*$. By Cohen's factorization theorem, $a=xby$ for some $x,b,y\in A$. Then
\begin{equation}\label{eq:bill}
\begin{aligned}
\pair{\Delta^{**}(M)\cdot a}{\psi}
 & = \pair{x\cdot\Delta^{**}(M)\cdot by}{\psi}  & \quad\text{(using \eqref{eq:boyle})} \\
 & = \pair{\Delta^{**}(M)}{by\cdot\psi\cdot x} \\
 & =  \pair{M}{\Delta^*(by\cdot\psi\cdot x)}. 
\end{aligned}
\end{equation}
But $\Delta^*(by\cdot\psi \cdot x) \in V$, and by its definition $M\vert_V = \imath^{-1}(F)\vert_V = F$, so
\begin{equation}\label{eq:ben}
\begin{aligned}
\pair{M}{\Delta^*(by\cdot\psi\cdot x)}
 & = \pair{F}{\Delta^*(by\cdot\psi\cdot x)}  \\
 & = \pair{\Delta^{**}(F)}{by\cdot\psi\cdot x} \\
 & = \pair{\Delta^{**}(F)\cdot b}{y\cdot \psi\cdot x} \\
 & = \pair{y\cdot\psi\cdot x}{b} , \\
\end{aligned}
\end{equation}
where the last equality holds since $\Delta^*(y\cdot\psi\cdot x) \in V$ and $F$ is a $V$-virtual diagonal.
Combining \eqref{eq:bill} and \eqref{eq:ben} gives $\pair{\Delta^{**}(M)\cdot a}{\psi} = \pair{\psi}{a}$.  Thus $M$ is a virtual diagonal for $A$.
\end{proof}

Recall that $\FA$ is Connes-amenable if and only if $A$ has a \WAP-virtual diagonal (Theorem~\ref{t:WAPVD-CA}), while by the previous proposition, $A$ is amenable if and only if it has a b.a.i.\ and an $(\AA)^*_{\ess}$-virtual diagonal.
We are therefore led to ask how the spaces $\AWAPA((\AA)^*)$ and $(\AA)^*_{\ess}$ are related.

\begin{lem}\label{l:biWAP-is-neounital}
Let $A$ be a Banach algebra with a~b.a.i. Let $X$ be a neo-unital $A$-bimodule. Then $\AWAPA(X^*)$ is neo-unital {\upshape(}and in particular is contained in $X^*_{\ess}${\upshape)}.
\end{lem}

\begin{proof}
It suffices to prove that $\AWAPA(X^*)_{\ess} =\AWAPA(X^*)$, since essential $A$-bi\-modules are neo-unital (by Lemma~\ref{l:ess-part}, or a direct argument with Cohen's factorization theorem). 

Let $\phi\in \AWAPA(X^*)$, and let $(e_i)$ be a b.a.i.~for~$A$, with norm $\leq C$ say. Then $\{e_i\cdot\phi\}$ is relatively weakly compact in $X^*$, so by passing to a subnet if necessary, we may assume that there exist $\psi\in X^*$ such that $e_i\cdot\phi \to \psi$ weakly in $X^*$. On the other hand, since $X$ is neo-unital, $e_i\cdot\phi\to\phi$  \wstar\ in $X^*$. Hence $\phi=\psi$, showing that $\phi$ belongs to the weak closure of $\{e_i\cdot\phi\}$. Taking convex combinations and using Mazur's theorem, it follows that $\phi$ belongs to the norm closure of $\{a\cdot\phi \st a\in A, \norm{a}\leq C\}$.

By a similar argument, considering the net $(\phi\cdot e_i)$, we see that $\phi$ belongs to the norm closure of $\{\phi\cdot b \st b\in A, \norm{b}\leq C\}$.
Therefore $\phi$ belongs to the essential part of $\AWAPA(X^*)$, as required.
\end{proof}

\begin{rem}
The proof of Lemma~\ref{l:biWAP-is-neounital} is probably folklore, as it is a straightforward generalization of the case $A=X=\LG$.
Indeed, it has already been noted elsewhere in the literature: see \cite[Proposition 3.12]{DalLau_Beurling} and \cite[Lemma 3.1]{Daws_Diss10}, or \cite[Proposition 3.3(b)]{Lau_CM87} for a one-sided version. We have kept the proof here since it is fairly short and instructive. 
\end{rem}

We can now try to prove that in certain cases, Connes-amenability of $\FA$ implies amenability of $A$. To illustrate the method, we consider an atypically simple case: namely the semigroup algebra $\ell^1(\Nmin)$, where $\Nmin$ denotes the set of natural numbers equipped with the product $(m,n)\mapsto \min(m,n)$. The following result is \cite[Theorem 7.6]{Daws_DBA2}.

\begin{thm}[Daws]\label{t:Daws_eg}
Let $A=\ell^1(\Nmin)$. Then $\FA$ is not Connes-amenable.
\end{thm}

The proof in \cite{Daws_DBA2} uses results linking Connes-amenability to a suitable notion of injectivity for representations of dual Banach algebras. We can give an alternative proof, using the following observation.

\begin{lem}\label{l:nice-WAP}
Let $A=\ell^1(\Nmin)$.
\begin{YCnum}
\item
The sequence $(\delta_n)_{n\geq 1}$ is a b.a.i.~for the algebra $A$;
\item
$(\AA)^*_{\ess}  \subseteq \AWAPA((\AA)^*)$.
\end{YCnum}

\end{lem}

\begin{proof}
Part (i) is straightforward, since
\[ \norm{\delta_n a - a} = \norm{a\delta_n - a} = \norm{ \sum_{k\geq n+1} a_k (\delta_n-\delta_k) } \leq 2 \sum_{k\geq n+1} \abs{a_k} \to 0 \quad\text{as $n\to\infty$}. \]
For part (ii), let $\Psi\in (\AA)^*_{\ess}$. Then (Cohen factorization) there exists $\Phi\in (\AA)^*$ and $a,b\in A$ such that $\Psi= a\cdot\Phi\cdot b$.
By part (i), $\delta_n a \to a$ as $n\to\infty$. Hence, the left $\Nmin$-orbit $\{ \delta_n \cdot \Psi \st n\in\Nmin \}$, when indexed as a sequence in the obvious way, converges in norm to $\Psi$. In particular this orbit is relatively compact, and by taking convex combinations we see that the set $\{ c\cdot \Psi \st c\in A, \norm{c}_1 \leq 1\}$ is also compact, so in particular is weakly compact. Similarly, $\{ \Psi\cdot c \st a\in A, \norm{c}_1 \leq 1\}$ is (weakly) compact, which completes the proof of~(ii).
\end{proof}

\begin{proof}[Proof of Theorem~\ref{t:Daws_eg}]
We argue by contradiction. Suppose $\FA$ is Connes-amenable. Then by Theorem~\ref{t:WAPVD-CA}, Proposition~\ref{p:near-VD} and Lemma~\ref{l:nice-WAP}, $A=\ell^1(\Nmin)$ would be amenable. But $A$ admits finite-dimensional quotients with arbitrarily large amenability constants (see \cite[Theorem 10]{DuncNam}), and this gives the desired contradiction.
\end{proof}

Naturally we would like to carry out similar arguments for other Banach algebras~$A$: in particular, as promised in the introduction, for $A=\LG$ when $G$ is a locally compact group~$G$. 
This case is much harder than $\ell^1(\Nmin)$, because in general
 the inclusion of $\AWAPA((\AA)^*)$ into $(\AA)^*_{\ess}$ is not surjective, as can be seen even for the case $A=\ell^1(\Z)$. Nevertheless, the following is true.

\begin{thm}\label{t:various-diag}
Let $G$ be a locally compact group. The following are equivalent, and characterize amenability of $G$:
\begin{YCnum}
\item\label{li:has-WAPVD}
 $\LG$ has a \WAP-virtual diagonal;
\item\label{li:has-UCVD}
 $\LG$ has an $\LIGG_{\ess}$-virtual diagonal;
\item\label{li:has-VD}
 $\LG$ has a virtual diagonal.
\end{YCnum}
\end{thm}

Note that the implication \ref{li:has-WAPVD}$\implies$\ref{li:has-VD}, combined with the easy direction of Theorem~\ref{t:WAPVD-CA}, gives another proof that Connes-amenability of $\F(\LG)=\WAP(G)^*$ implies amenability of~$\LG$, cf.~Theorem~\ref{t:BEJ_VR} and the remarks after~it.

Some parts of Theorem~\ref{t:various-diag} are easily proved from what we already know:
 \ref{li:has-UCVD}$\implies$\ref{li:has-VD} follows by taking $A=\LG$ in Proposition~\ref{p:near-VD};
 while
 \ref{li:has-VD}$\implies$\ref{li:has-WAPVD} follows by taking a virtual diagonal for $\LG$ and restricting it to $\LWAPL$.
To complete the proof of Theorem~\ref{t:various-diag}, we need to show that \ref{li:has-WAPVD} implies~\ref{li:has-UCVD}. This is the hard part, and will be addressed in the next section.

\begin{rem}\label{r:what-are-they}
The reader may wish to have a more concrete description of $\LIGG_{\ess}$ and $\LWAPL$ as concrete subspaces of $\LIGG$. This is not too difficult, and one obtains descriptions similar to the known identifications
\[  \LIG_{\ess} = \UC(G) \quad,\quad \LGWAPLG(\LIG) = \WAP(G). \]
Here, $\UC(G)$ denotes the space of uniformly continuous bounded functions on $G$, sometimes denoted in the literature by $\UCB(G)$.
 Since we do not need these descriptions to complete the proof of Theorem~\ref{t:various-diag}, we have deferred them to Section~\ref{s:ess-and-WAP-of-LIGG}.
\end{rem}

\end{section}

\begin{section}{Obtaining a $\LIGG_{\ess}$-diagonal from a \WAP-virtual diagonal}
\label{s:anabasis}
Given a \WAP-virtual diagonal for $\LG$, we wish to produce a $\LIGG_{\ess}$-virtual diagonal. Our approach uses some results of Runde from the article \cite{Run_MG1}, and adapts some of his arguments. Thus, some of what we do is a recasting of his work into a new mold. However, since it is not always easy to find what we need, stated explicitly in the form we need, we shall repeat some of the necessary details, albeit with some technical modifications.

\begin{notn}
For convenience of notation, for the rest of this section we shall denote $\LIGG_{\ess}$ by~$\LUCR$, and $\LWAPL$ by~$\cW$.
Also, to reduce potential confusion, we shall denote the pointwise products in $\LIG$ and in $\LIGG$ by $\bullet$, and likewise we shall use $\bullet$ to denote certain adjoint module actions that are induced from pointwise product.
(The reader should beware that the symbol $\bullet$ is also used in \cite{Run_MG1}, but has a different meaning.)
We reserve the symbol $\cdot$ for the action of $\LG$ on various submodules of $\LIG$ and $\LIGG$. Although it might be more natural to use convolution notation, our choice reflects some hope of applying the arguments here to other Banach algebras.

In view of our standing convention that ``bimodule'' really means ``Banach bimodule'', we use the terminology ``$G_d$-bimodule'' to mean a Banach space that is a bimodule for $G$ in the purely algebraic sense.
This is distinct from the standard notion in the literature of a ``Banach $G$-bimodule'', where one requires the orbit maps to be continuous as functions on $G$.
\end{notn}

\subsection{Submodules for the actions of $G_d$, $M(G)$ and $\LG$.}
We need to consider not only the $\LG$-bimodule actions on $\LIGG$ and on its subspaces $\cU$ and $\cW$, but also the induced $G_d$-bimodule actions.
To fix notation, let us briefly review how this works.

Since $\LG$ has a b.a.i., there are left and right actions of the measure algebra $M(G)$ on the space $\LG$, which extend the usual left and right multiplication action of $\LG$ on itself.
 This makes $\LGG=\LG\ptp\LG$ an $M(G)$-bimodule, in a way that extends the natural action of $\LG$ on $\LG\ptp \LG$, and hence by duality $\LIGG=\LGG^*$ becomes an $M(G)$-bimodule.
Note that with this bimodule structure, the left action of $M(G)$ on $\LIGG$ acts on the \emph{second} variable of $G\times G$, and the right action acts on the \emph{first variable} in $G\times G$.

Via the inclusion of $\ell^1(G_d)$ into $M(G)$ as the subalgebra of discrete measures, $\LIGG$  is then a $G_d$-bimodule in a way that is compatible with the $\LG$-bimodule structure. Explicitly, one can check using the usual formulas for convolution of measures and using the definition of adjoint actions on the dual of a bimodule, that for each $x\in G$ we have
\begin{equation}\label{eq:G-G-action}
\begin{aligned}
 (\delta_x \cdot h )(s,t) & \overset{\rm l.a.e.}{=} h(s,tx) \quad\qquad (h\in \LIGG; s,t,\in G), \\
 (h \cdot \delta_x )(s,t) & \overset{\rm l.a.e.}{=} h(xs,t) \quad\qquad (h\in \LIGG; s,t,\in G), \\
\end{aligned}
\end{equation}
where ``l.a.e.'' stands for ``locally almost everywhere''.

We will consider various submodules of $\LIGG$. This requires the following caveat: if $X$ is a general $M(G)$-bimodule, it is both an $L^1(G)$-bimodule and a $G_d$-bimodule, but not every sub-$G_d$-module of $X$ will be a sub-$\LG$-bimodule. (For instance, take the canonical copy of $\ell^1(G_d)$ inside $X=M(G)$.)
Fortunately, things are fine if we are working inside a \emph{neo-unital} $\LG$-bimodule~$X$.
By a standard general procedure (see, e.g. Sections 1.d and 2 of \cite{BEJ_CIBA}), $X$ becomes a $M(G)$-bimodule; moreover, every sub-$G_d$-bimodule $V$ is automatically a sub-$M(G)$-bimodule, and hence a sub-$\LG$-bimodule. This follows because $V$ is neo-unital and $\ell^1(G_d)$ is dense in $M(G)$ with respect to the strict topology.

\subsection{The key space.}
Recall that $\Delta$ denotes the multiplication map $\LGG\to \LG$ (i.e.~convolution), so that $\Delta^*:\LIG\to\LIGG$ is given by $\Delta^*(f)(s,t) =f(st)$ for every $s,t\in G$.

\begin{dfn}
Let $\cI$ denote the closed ideal in $\LUCR$ generated by the subalgebra $\Delta^*(C_0(G))$, that is,
\[
\cI \defeq \overline{ \lin \{ \Delta^*(f)\bullet h \st f\in C_0(G) , h\in \LUCR \} }.
\]
It is clear from its definition, and from the formulas \eqref{eq:G-G-action}, that $\cI$ is a sub-$G_d$-bimodule of $\LUCR$. Therefore, by the previous remarks, it is a (neo-unital) sub-$\LG$-bimodule of $\LUCR$ and a sub-$M(G)$-bimodule of $\LUCR$.
\end{dfn}

The next lemma, which is crucial, collates several parts of results in \cite{Run_MG1}.
\begin{lem}[Runde]\label{l:LUCSC_0}
$\cI\subseteq\cW$.
\end{lem}

\begin{proof}
Let $\LUC\SC_0(\GGop)$ be the closed subspace of $\LIGG$ that is defined in \cite[Definition 4.2]{Run_MG1}. By [Lemma~5.3, \ibid] it is an ideal in $\LUCR$, and by [Theorem~4.6(iii), \ibid] it contains $\Delta^*(C_0(G))$; therefore it contains $\cI$.

Now fix $h\in \cI$. Since $h\in \LUC\SC_0(\GGop)$, \cite[Lemma 4.4]{Run_MG1} implies that the $G$-orbits
 $\{\delta_x\cdot h\st x\in G\}$ and $\{h\cdot \delta_x \st x\in G\}$ are relatively weakly compact.
Since absolutely convex hulls of weakly compact sets are weakly compact, it follows that the sets $C^L_h\defeq \{ a\cdot h \st a \in \ell^1(G_d), \norm{a}\leq 1\}$ and $R^L_h\defeq \{ h\cdot a \st a\in \ell^1(G_d), \norm{a}\leq 1\}$ are relatively weakly compact.

By standard approximation results for measures, given any $\mu\in M(G)$ there is a net $(a_i)\subset \ell^1(G_d)$ with $\norm{a_i}\leq \norm{\mu}$ for all $i$ and $a_i\to \mu$ in the strict topology of $M(G)$.
See e.g.~\cite[Lemma 1.1.3]{Gre}.
 As $\cI$ is neo-unital as an $\LG$-bimodule, it follows that $\norm{a_i\cdot h - \mu\cdot h}\to 0$. Hence $\{\mu\cdot h \st \mu\in M(G), \norm{\mu}\leq 1\}$ is contained in the norm closure of $C^L_h$, and so is relatively weakly compact. Repeating the argument with left and right reversed, we see that $\{h\cdot\mu \st \mu\in M(G), \norm{\mu}\leq 1\}$ is also relatively weakly compact. As $\LG\subseteq M(G)$ this shows $h\in \cW$.
\end{proof}

\begin{rem}
In the proof of Lemma~\ref{l:LUCSC_0}, we passed from relatively weakly compact $G$-orbits to weak compactness of orbit maps $M(G)\to \cI$. The same argument works more generally: see the proof of Theorem~\ref{t:WAP-of-LIGG} in the next section. We chose to prove the special case first, since it is slightly easier: see Remark~\ref{r:not-so-easy}.
\end{rem}

\subsection{The key technical results.}
The next two results (Lemma~\ref{l:G-fixed} and Proposition~\ref{p:G-retract}) are based very closely on ideas from the proof of~\cite[Theorem 5.4]{Run_MG1}. However, it seems clearer to isolate and state them in the form we require, rather than to explain in piecemeal fashion how one modifies the proof of that theorem.

Since $\cI$ is a commutative $\Cst$-algebra, its bidual $\cI^{**}$ is unital, with identity element $P$, say. We equip $\cI^{**}$ with the natural $M(G)$-bimodule structure induced from that of~$\cI$.

\begin{lem}[Runde]\label{l:G-fixed}
Let $x\in G$. Then $\delta_x\cdot P=P=P\cdot \delta_x$. Consequently, if $(v_i)$ is a bounded net in $\cI$ that converges \wstar\ in $\cI$ to $P$, then so are $(\delta_x\cdot v_i)$ and $(v_i\cdot \delta_x)$.
\end{lem}

For convenience we give the proof.

\begin{proof}
Given $x\in G$, define $L_x: \cI\to \cI$ and $R_x:\cI\to\cI$ by $L_x(h)=\delta_x\cdot h$ and $R_x(h)=h\cdot\delta_x$\/, for $h\in \cI$. Both $L_x$ and $R_x$ are algebra automorphisms of $\cI$, with inverses $L_{x^{-1}}$ and $R_{x^{-1}}$ respectively.
The second adjoint of an algebra automorphism is always an automorphism of the second dual (with respect to either Arens product). Thus $L_x^{**}$ and $R_x^{**}$ are automorphisms of the unital algebra $\cI^{**}$, and so must fix the identity element of $\cI^{**}$, which is~$P$. This proves the first part.

Now fix $x\in G$, and suppose $(v_i)\subset\cI$ with $v_i\to P$ $\wstar$. Then for any $\phi\in \cI^*$ we have
\[ \begin{aligned} \pair{\phi}{\delta_x\cdot v_i}_{\cI^*-\cI}
 = \pair{\phi\cdot\delta_x}{v_i}_{\cI^*-\cI}
 & \to \pair{P}{\phi\cdot\delta_x}_{\cI^{**}-\cI^*} \\
 & = \pair{\delta_x\cdot P}{\phi}_{\cI^{**}-\cI^*} = \pair{P}{\phi}_{\cI^{**}-\cI^*}\;,
\end{aligned} \]
and so $\delta_x\cdot v_i\to P$ \wstar. A similar argument shows that $(v_i\cdot\delta_x)\to P$ \wstar.
\end{proof}

Consider the natural right action of $\LUCR$ on $\cI^*$ (the adjoint of the left action $\LUCR\times\cI\to\cI$ that is given by multiplication of functions).
This gives rise to a bounded, bilinear map $\cI^{**}\times \cI^* \to \LUCR^*$, denoted by $(F,\psi)\mapsto F_\psi$ for $F\in \cI^{**}$ and $\psi\in \cI^*$, and defined by
$\pair{F_\psi}{h} = \pair{F}{\psi\bullet h}$ for all $\psi\in\cI^*$ and $h\in\cI$.
(We are performing the first two stages of the canonical Arens extension of the module action of $\LUCR$ on~$\cI$.)


Recall that  $\Delta^*(\UC(G)) = \Delta(\LIG_{\ess})\subseteq \LIGG_{\ess} \equiv \cU$, giving us the adjoint map $\Delta_{\UC}:\cU^* \to \UC(G)^*$. Let $\imath: M(G)\to \UC(G)^*$ be the natural inclusion map, defined by
\begin{equation}
\pair{\imath(\mu)}{f}_{\UC(G)^*-\UC(G)} \defeq \int_G f\,d\mu \qquad (\mu\in M(G), f\in \UC(G)).
\end{equation}

\begin{prop}\label{p:G-retract}
There is a $G_d$-bimodule map $S:\cI^*\to \LUCR^*$ which makes the following diagram commute.
\begin{equation}\label{eq:compatible}
 \begin{diagram}[tight,height=2.5em,width=4.5em]
C_0(G)^* & \lTo^{\Delta_{C_0}} & \cI^* \\
\dTo^{\imath} & & \dTo_S \\
\UC(G)^* & \lTo_{\Delta_{\UC}} & {\LUCR}^*
\end{diagram}
\end{equation}
\end{prop}

\begin{proof}
Let $P$ be the identity element of $\cI^{**}$, and define $S(\phi)=P_\phi\in \LUCR^*$ for each $\phi\in \cI^*$. Thus, for any $h\in \LUCR$, and \emph{any} bounded net $(v_i)\subset \cI$ which converges to $P$ in the $\wstar$-topology of $\cI^{**}$, we have
\[
 \pair{S(\phi)}{h}_{\LUCR^*-\LUCR} = \lim_i \pair{\phi}{h\bullet v_i}_{\cI^*-\cI}\;.
\]

Fix $x\in G$. Then
\[ \begin{aligned}
\pair{S(\delta_x\cdot \phi)}{h}_{\LUCR^*-\LUCR}
 & = \lim_i \pair{\delta_x\cdot \phi}{h\bullet v_i}_{\cI^*-\cI} \\
 & = \lim_i \pair{\phi}{(h\bullet v_i)\cdot\delta_x}_{\cI^*-\cI} \\
 & = \lim_i \pair{\phi}{(h\cdot\delta_x)\bullet (v_i\cdot\delta_x)}_{\cI^*-\cI} 
 & = \lim_i \pair{\phi\bullet(h\cdot\delta_x) }{ v_i\cdot\delta_x}_{\cI^*-\cI} \;.
\end{aligned}\]
But by Lemma~\ref{l:G-fixed}, the net $(v_i\cdot\delta_x)$ also converges to $P$ in the \wstar-topology of $\cI^{**}$. Hence
\[
\lim_i \pair{\phi\bullet(h\cdot\delta_x)}{ v_i\cdot\delta_x }_{\cI^*-\cI} 
= \pair{S(\phi)}{h\cdot\delta_x}_{\LUCR^*-\LUCR} 
= \pair{\delta_x\cdot S(\phi)}{h}_{\LUCR^*-\LUCR} \,.
\]
Combining the previous equations, we see that $S$ is a left $G_d$-module map. A similar argument, with left and right switched, shows that $S$ is a right $G_d$-module map.


We need to show that the diagram in \eqref{eq:compatible} commutes. Let $\phi\in \cI^*$ and let $\mu=\Delta_{C_0}(\phi)\in C_0(G)^* = M(G)$. Then
\[ \pair{\imath\Delta_{C_0}(\phi)}{f}_{\UC(G)^*-\UC(G)} = \int_G f \,d\mu \qquad(f\in \UC(G)). \]
Now let $(u_i)\subset C_0(G)$ be a b.a.i.\ and let $f\in \UC(G)$.
The net $(\Delta(u_i))$ is a b.a.i.\ for $\Delta^*(C_0(G))$, and hence also for $\cI$.
Passing to a subnet if necessary, we may assume that $\Delta^*(u_i)\to P$ \wstar\ in $\cI^{**}$. Therefore,
\[ \begin{aligned}
\pair{S(\phi)}{\Delta^*(f)}_{\LUCR^*-\LUCR}
 & = \lim_i \pair{\phi}{\Delta^*(f)\bullet \Delta^*(u_i)}_{\cI^*-\cI}  \\
 & = \lim_i \pair{\phi}{\Delta^*(f \bullet u_i)}_{\cI^*-\cI} \\
 & = \lim_i \pair{\Delta_{C_0}(\phi)}{f\bullet u_i}_{C_0(G)^*-C_0(G)} \\
 & = \lim_i \int_G f\bullet u_i \,d\mu 
 & = \int_G f\,d\mu\;,
\end{aligned} \]
where the last equality holds because $\mu$ is a \emph{finite} Radon measure on $G$. (Note that in general, $u_i\bullet f$ does not converge to $f$ in norm.) Thus,
\[ \pair{\Delta_{\UC}(S(\phi))}{f}_{\UC(G)^*-\UC(G)} = \pair{\imath(\Delta_{C_0}(\phi))}{f}_{\UC(G)^*-\UC(G)} \]
for all $\phi\in\cI^*$ and $f\in\UC(G)$, as required.
\end{proof}

We would really like the map $S$ constructed in Proposition~\ref{p:G-retract} to be not just a $G_d$-bimodule map, but an $L^1(G)$-bimodule map. It is not clear to us if this is always possible, since even though $\cI^*$ is an $\LG$-bimodule, it is usually not neo-unital.
This may be related to an issue left unaddressed\footnotemark\ in~\cite{Run_MG1}.
\footnotetext{At the top of p.~655 in \cite{Run_MG1}, it is claimed that a certain embedding of $\mathcal{LUC}(\GGop)$ into $\mathcal{LUCSC}_0(\GGop)^{**}$ is a $M(G)$-bimodule map. This appears to need extra justification. However, at that point in \cite{Run_MG1} one only requires that this embedding is an $\ell^1(G_d)$-bimodule map, which is what is shown.}
However, we can sidestep this obstacle using the following lemma.

\begin{lem}\label{l:G-diagonal}
Let $\cV$ be a neo-unital $\LG$-bimodule, so that $\cV$ and $\cV^*$ are $M(G)$-bimodules, and let $\tils\in\cV^*$. Then
the following are equivalent:
\begin{YCnum}
\item\label{li:G-diag}
$\delta_x\cdot \tils = \tils\cdot \delta_x$ for all $x\in G$;
\item\label{li:M(G)-diag}
$\mu\cdot\tils = \tils\cdot\mu$ for all $\mu\in M(G)$.
\item\label{li:L^1(G)-diag}
$a\cdot\tils=\tils\cdot a$ for all $a\in \LG$.
\end{YCnum}
\end{lem}

\begin{proof}[Proof that \ref{li:G-diag}$\implies$\ref{li:M(G)-diag}.]
Since $\cV$ is neo-unital for $\LG$, a short calculation shows the orbit maps of $\tils$ are strict-to-\wstar\ continuous as maps $M(G)\to \cV^*$. Moreover, the natural embedding $\ell^1(G_d)\hookrightarrow M(G)$ has strictly dense range. So \ref{li:M(G)-diag} follows from \ref{li:G-diag} by density.

\paragraph{\it Proof that \ref{li:M(G)-diag}$\implies$\ref{li:L^1(G)-diag}.} This is trivial.

\paragraph{\it Proof that \ref{li:L^1(G)-diag}$\implies$\ref{li:G-diag}.}
Suppose \ref{li:L^1(G)-diag} holds, and let $\phi\in \cV$. We have
\begin{equation}\label{eq:ennui}
\tag{$**$}
\pair{\tils}{\phi\cdot a} = \pair{a\cdot\tils}{\phi} = \pair{\tils\cdot a}{\phi} = \pair{\tils}{a\cdot\phi}\quad(a\in \LG).
\end{equation}

Now let $(e_i)$ be a b.a.i.~in $\LG$. Since $\cV$ is a neo-unital $\LG$-bimodule, for any $\tilx\in \cV^*$ we have
$\wstar\lim_i \tilx\cdot e_i = \wstar\lim_ie_i\cdot\tilx=\tilx$.
Hence
\[ \begin{aligned}
\pair{\mu\cdot\tils}{\phi}_{\cV^*-\cV}
  = \lim_i \pair{\mu\cdot\tils}{\phi\cdot e_i}_{\cV^*-\cV} 
 & = \lim_i \pair{\tils}{\phi\cdot e_i\cdot\mu}_{\cV^*-\cV} \\
 & = \lim_i \pair{\tils}{(e_i\cdot\mu)\cdot\phi}_{\cV^*-\cV} & \quad\text{(by \eqref{eq:ennui})} \\
 & = \pair{\tils}{\mu\cdot\phi}_{\cV^*-\cV}
 \quad= \pair{\tils\cdot\mu}{\phi} \;.
\end{aligned} \]
Thus \ref{li:M(G)-diag} holds. This completes the proof of Lemma~\ref{l:G-diagonal}.
\end{proof}

\subsection{The proof of Theorem~\ref{t:various-diag}, \ref{li:has-WAPVD}$\implies$\ref{li:has-UCVD}.}
Let $\tilm\in \F_{\LG}(\LGG)= \cW^*$ be a \WAP-virtual diagonal for $\LG$. Then for any $a\in\LG$ and $\phi\in\WAP(\LIG)$,
$a\cdot\tilm =\tilm\cdot a$ and
\[
\pair{\DWAP(\tilm)\cdot a}{\phi}_{\WAP(\LIG)^*-\WAP(\LIG)}
 = \pair{\phi}{a}_{\WAP(\LIG)-\LG} \;.
\]

By Lemma~\ref{l:LUCSC_0}, $\Delta^*(C_0)\subseteq \cI \subseteq \cW$, and so $\tilm$ can be restricted to a functional on~$\cI$. Let $S:\cI^*\to \LUCR^*$ be the map provided by Proposition~\ref{p:G-retract}, and put $\tiln = S( \tilm\vert_{\cI} )\in \LUCR^*$. We will show $\tiln$ is a $\LUCR$-virtual diagonal.

Let $x\in G$. By Lemma~\ref{l:G-diagonal}, $\delta_x\cdot (\tilm\vert_{\cI}) = (\tilm\vert_{\cI})\cdot\delta_x$. By Proposition~\ref{p:G-retract}, $S$ is a $G_d$-bimodule map, and so
\[ \delta_x\cdot \tiln = S\left( \delta_x\cdot (\tilm\vert_{\cI})\right) = S\left( (\tilm\vert_{\cI})\cdot\delta_x\right)  = \tiln\cdot\delta_x. \]
Using Lemma~\ref{l:G-diagonal} in the other direction, we have $a\cdot\tiln=\tiln\cdot a$ for all $a\in\LG$.

Let $\mu = \Delta_{C_0}(\tilm\vert_{\cI})\in M(G)$.
Since $\tilm$ is a \WAP-virtual diagonal, $\mu\cdot a =a$ as elements of $\LG$, for all $a\in \LG$, hence $\mu=\delta_e$. 
Therefore, by Proposition~\ref{p:G-retract},
\[ \Delta_{\UC}(\tiln) = \imath(\Delta_{C_0}(\tilm\vert_{\cI})) =\imath(\delta_e). \]
For each $a\in \LG$, $\imath(\delta_e)\cdot a = a$ as elements of $\UC(G)^*$: one can either check this by a direct calculation, or else observe that $\imath(\delta_e)\cdot a$ coincides with the Arens-type product of $\imath(\delta_e)$ and $a$ in the algebra $\UC(G)^*$, which is known to be unital with identity $\imath(\delta_e)$. Hence
$\pair{\Delta_{\UC}(\tiln)\cdot a}{\psi}_{\UC(G)^*-\UC(G)} = \pair{\psi}{a}_{\UC(G)-\LG}$
for all $a\in \LG$ and all $\psi\in \UC(G)$,
which, in view of Lemma~\ref{l:patch-ess}, completes the proof that $\tiln$ is a $\LUCR$-virtual diagonal.
\hfill$\Box$
\end{section}

\begin{section}{The essential and 2-sided-$\WAP$ parts of $L^\infty(\GG)$}
\label{s:ess-and-WAP-of-LIGG}
In this section we obtain natural ``bivariate'' counterparts  (Theorems~\ref{t:ess-of-LIGG} and \ref{t:WAP-of-LIGG} below) of the existing, standard characterizations of $\LIG_{\ess}$ and $\LGWAPLG(\LIG)$, see Remark~\ref{r:what-are-they}.
As in the previous section, we denote the product in the algebra $\LG$ by juxtaposition and the induced action on $\LIG$ and $\LIGG$ by a dot, reserving $\bullet$ for the pointwise product of functions.
In order to fix notation and remove potential ambiguity, we include the following definition.

\begin{dfn}
Given a locally compact group $H$, we define $\LUC(H)$ to be the space of all $f\in \CB(H)$ for which the map $f\mapsto f\cdot\delta_x$, defined by $(f\cdot\delta_x)(t)\defeq f(xt)$, is continuous as a function $H\mapsto (\CB(H),\norm{\cdot}_\infty)$. {\bf Warning:}~in some older sources, this space is denoted by $\operatorname{RUC}(H)$ or $\UC_r(H)$, but the present notation appears to be more commonly used in recent work.
\end{dfn}

It is a classical result that $\LUC(H)$ coincides with $L^\infty(H)\cdot L^1(H)$. Armed with this, we can now state our characterization of $\LIGG_{\ess}$. The reader may find it helpful to recall that when we consider $\LG\cdot \LIGG\cdot \LG$, the left action operates on the second variable of $\GG$, and the right action on the first variable, so that this space naturally contains $(\LIG\cdot\LG)\otimes(\LG\cdot\LIG)$.

\begin{thm}\label{t:ess-of-LIGG}
 $\LIGG_{\ess}= \LUC(\GGop)$.
\end{thm}

\begin{proof}
For this proof, denote the action of the group algebra $L^1(\GGop)$ on its dual, $L^\infty(\GGop)$, by $\odot$. Then $L^\infty(\GGop)\odot L^1(\GGop)$ coincides as a subspace of $L^\infty(\GGop)$ with $\LUC(\GGop)$, so it suffices to prove that
\begin{equation}\label{eq:pedantic}
 L^\infty(\GGop)\odot L^1(\GGop) = \LIGG_{\ess} \quad\text{as subspaces of $\LIGG$}.
\tag{$\ddagger$}
\end{equation}

The basic idea is straightforward. The opposite algebra $\LG^{\rm op}$ is isomorphic to the group algebra $L^1(\Gop)$, where $\Gop$ is $G$ equipped with the opposite multiplication. Therefore, since $\LG$-bimodules can be regarded as one-sided modules over the enveloping algebra $\LG\ptp L^1(G)^{\rm op}$, they can (by a suitable change of variables) be regarded as one-sided modules over the group algebra $L^1(\GGop)$. However, since we want to identify both sides of \eqref{eq:pedantic} as concrete spaces of functions on $G\times G$, not just as abstractly isomorphic modules, we proceed in some detail.

Let $\lambda$ be the chosen left Haar measure on $G$ and $\Delta$ the modular function, so that $\Delta^{-1} \bullet \lambda$ is a \emph{left} Haar measure on $G^{\rm op}$.  Direct calculation, using properties of the modular function, shows that the map
$\theta: \LG^{\rm op}\to L^1(\Gop)$, $f \mapsto f\bullet\Delta$,
is an isometric isomorphism of Banach algebras.
Now, let $a,b,u,v\in \LG$ and let $F\in \LIGG=L^\infty(\GGop)$. Then $b\tp\theta(a)$ and $u\tp\theta(v)$ lie in $L^1(\GGop)$, and
\begin{eqnarray*}
 & & \pair{ F\odot  (b\tp \theta(a)) }{u\tp\theta(v)}_{L^\infty(\GGop)-L^1(\GGop)} \\
 & = & \pair{F}{bu \tp \theta(a)\theta(v)}_{L^\infty(\GGop)-L^1(\GGop)} \\
 & = & \pair{F}{bu \tp \theta(va)}_{L^\infty(\GGop)-L^1(\GGop)} \\
 & = & \pair{F}{bu\tp va}_{\LIGG-\LGG} \\
 & = & \pair{a\cdot F\cdot b}{u\tp v}_{\LIGG-\LGG}\;.
\end{eqnarray*}
By continuity, we therefore have
$F\odot (b\tp \theta(a) ) = a\cdot F \cdot b$
for all $a,b\in \LG$ and all $F\in \LIGG$,
from which~\eqref{eq:pedantic} follows.
\end{proof}

We now turn to characterizing $\LWAPL$. First, we have a general lemma: it could have gone in Section~\ref{s:diag-subsp}, but it is only needed here. 

\begin{lem}\label{l:multiplier-WAP}
Let $A$ be a Banach algebra with a b.a.i.~and let $M(A)$ be its multiplier algebra. Let $X$ be a neo-unital $A$-bimodule, and regard it as an $M(A)$-bimodule in the natural way (see e.g.~\cite[\S1.d]{BEJ_CIBA} or \cite[Theorem 3.2]{Daws_Diss10} for details). Then
$\AWAPA(X^*)= \basewap{M(A)}(X^*)$.
\end{lem}

\begin{proof}
By the way we define the action of $M(A)$ on $X$ and hence on $X^*$, the inclusion\breakinscand\
 $\basewap{M(A)}(X^*)\subseteq \AWAPA(X^*)$ is straightforward. Conversely, let $\phi\in \AWAPA(X^*)$, and let $R_\phi^{M(A)} : M(A)\to X^*$ be the orbit map $F\mapsto F\cdot\phi$. Since $\AWAPA(X^*)$ is $A$-neo-unital by  Lemma~\ref{l:biWAP-is-neounital}, we have $\phi = a\cdot \psi$ for some $a\in A$ and $\psi\in \AWAPA(X^*)$. Therefore $R_\phi^{M(A)}$ can be factorized as $R^A_{\psi}\circ R^{M(A)}_a$,
using the fact that $A$ is a right ideal in $M(A)$. Since the second map is weakly compact, $R_\phi^{M(A)}$ is weakly compact. A similar argument, with left and right reversed,  shows that the orbit map $F\mapsto \phi\cdot F$ is also weakly compact. Thus $\phi\in\basewap{M(A)}(X^*)$ as required.
\end{proof}

Recall that $\ell^1(G_d)$ acts naturally on $\LIGG$ by left and right translations, with these actions arising from the natural $M(G)$-bimodule structure on $\LIGG$.

\begin{dfn}
Define $\GWAPG$ to be the set of all $f\in \CB(\GG)$ such that both the left and right $G$-orbits of $f$ are relatively weakly compact.
\end{dfn}

\begin{notn}
It is useful, as in the previous section, to write
\[ \cW \defeq \LGWAPLG (\LIGG).\]
Similarly, we write
$\cW_M \defeq \basewap{M(G)}(\LIGG)$ and
 $\cW_d \defeq \basewap{\ell^1(G_d)} (\LIGG)$.
\end{notn}

\begin{rem}\label{r:not-so-easy}
Since convex hulls of weakly compact sets are weakly compact, it follows from the definitions that $\GWAPG= \cW_d \cap \CB(\GG)$. What is not immediate from the definition of $\GWAPG$ is that it is contained in $\LIGG_{\ess}$. This makes the proof of the following result a little more tricky than one might hope; in particular, the argument used in the proof of Lemma~\ref{l:LUCSC_0} no longer suffices.
\end{rem}

\begin{thm}\label{t:WAP-of-LIGG}
The spaces $\GWAPG$,  $\cW$, $\cW_M$ and $\cW_d$ all coincide. In particular, they are all subspaces of $\CB(\GG)$.
\end{thm}

\begin{proof}
By Lemma~\ref{l:multiplier-WAP}, $\cW =  \cW_M$. Also, by Lemma~\ref{l:biWAP-is-neounital} and Theorem~\ref{t:ess-of-LIGG}, $\cW \subseteq \LUC(\GGop) \subseteq \CB(\GG)$. Hence
\[ \cW = \cW_M =\cW_M\cap \CB(\GG) \subseteq \cW_d\cap \CB(\GG) = \GWAPG \;,\]
the inclusion being a trivial consequence of the inclusion $\ell^1(G_d)\subset M(G)$, and the final equality holding by Remark~\ref{r:not-so-easy}. The only thing left to prove is that $\cW_d\subseteq \cW_M$.

Let $h \in \cW_d$. As $G\cdot h$ is relatively weakly compact,
 and convex hulls of weakly compact sets are weakly compact, the set $S=\{ a\cdot h \st a\in \ell^1(G_d), \norm{a}\leq 1\}$ is a relatively weakly compact subset of $\LIGG$.
When a subset of a dual space is relatively weakly compact, its weak and \wstar\ closures coincide. Therefore, $\overline{S}^{\wstar}$ is weakly compact.

Now, given $\mu \in M(G)$ with $\norm{\mu} \leq 1$,  there is a net $(f_i)$ in the unit ball of $\ell^1(G_d)$ such that $f_i \to \mu$ in the strict topology of $M(G)$: see, e.g.~\cite[Lemma 1.1.3]{Gre}.
It follows that
$\norm{ w \cdot f_i - w \cdot \mu} \to 0$ for all $w \in \LGG$
 (one can first prove this for $w$ an elementary tensor in $\LG\tp\LG$, and then observe that such tensors have dense linear span in $\LGG$).
Therefore $f_i \cdot h \to \mu \cdot h$ \wstar\ in $\LIGG$, so that
$\{ \mu \cdot h: \mu \in M(G), \ \|\mu \| \leq1\} \subseteq \overline{S}^{\wstar}$.
By the previous paragraph, this implies that the left orbit map $R^{M(G)}_h : M(G)\to \LIGG$ is weakly compact.

A similar argument, with left and right swapped, shows that the right orbit map $L^{M(G)}_h:M(G)\to \LIGG$ is also weakly compact. Thus $h \in  \cW_M$, as required.
\end{proof}

\end{section}

\begin{section}{Closing remarks and questions}
\label{s:closing}
Having developed the basic machinery of the functor $\F_A$, which associates to each  $A$-bimodule a canonical normal dual $\AWAPA(A^*)^*$-bimodule, it would be interesting to understand $\F_A$ in more detail. For instance, the proof of Lemma~\ref{l:surprising-extension} (and hence, the proof of Theorem~\ref{t:weakly-universal}) could have been made more transparent if we knew that $\F_A$ preserves short exact sequences of $A$-bimodules. A proof of this, or a class of counterexamples, would be desirable. One would also like to know what $\F_A$ does to projective, injective or flat $A$-bimodules: are they sent to normal dual modules that satisfy appropriate \emph{and non-artificial} notions of projectivity, injectivity or flatness for the category of normal dual $\FA$-bimodules?

The initial motivation for introducing and studying $\F_A$ was an attempt to find a systematic approach, for suitable classes of algebras, to the problem of whether Connes-amenability of $\AWAPA(A^*)$ implies amenability of~$A$. With this in mind, the methods of Section~\ref{s:anabasis} might be applicable, with suitable modifications, to the study of Connes-amenability for $\F(\cM_*)$ when $\cM$ is a commutative Hopf--von Neumann algebra: the cases $\cM=L^\infty(G)$ and $\cM=\ell^\infty(\Nmin)$ have been treated in the present paper.
The general setting is left for possible future research.

Moving away from topics related to amenability: the space $\GWAPG$ seems deserving of further study. In particular, while it contains $\WAP(\GG)=\WAP(\GGop)$, how much bigger is it in general? Also, it is easy to show that each $h\in\GWAPG$ is separately weakly almost periodic, in the sense that $h(\cdot,x)\in\WAP(G)$ and $h(x,\cdot)\in\WAP(G)$ for all $x\in G$. Does this characterize the functions in $\GWAPG$? We intend to return to these questions in other future work.
\end{section}

\section*{Acknowledgements}
The diagrams in this paper were prepared using Paul Taylor's \texttt{diagrams.sty} macros.

This project was initiated during a visit of the third author (RS) to Saskatoon in December 2010.
Further work was done during the Banach Algebras 2011 conference in Waterloo.
Some revisions were also made while the first and second authors (YC+ES) attended the meeting ``Harmonic analysis, operator algebras, and representations'' held at the CIRM, Luminy in October 2012.

The first author (YC) was partially supported by a New Faculty start-up grant from the University of Sask\-atchewan, and partially by NSERC Discovery Grant 402153-2011. He also thanks the Department of Mathematics and Statistics at Lancaster for its support while important revisions were made to the paper.
The second author (ES) was supported by NSERC Discovery Grant 366066-2009. The third author (RS) was supported by NSERC Discovery Grant 29844-2010.

Finally, the authors thank the referee for an attentive reading and several valuable, detailed suggestions, with particular thanks for the improvements to an earlier version of  Theorem~\ref{t:weakly-universal} and its proof.


\bibliography{wap-dba-paper}
\bibliographystyle{siam}

\vfill

\leftline{Yemon Choi}
\leftline{Department of Mathematics and Statistics}
\leftline{University of Saskatchewan}
\leftline{Saskatoon (SK), Canada S7N 5E6}

\smallskip
\leftline{\textbf{\textsf{Current address:}}}
\leftline{Department of Mathematics and Statistics}
\leftline{Lancaster University}
\leftline{Lancaster, United Kingdom LA1 4YF}
\leftline{\texttt{y.choi1@lancaster.ac.uk}}

\bigskip
\leftline{Ebrahim Samei}
\leftline{Department of Mathematics and Statistics}
\leftline{University of Saskatchewan}
\leftline{Saskatoon (SK), Canada S7N 5E6}
\leftline{\texttt{choi@math.usask.ca} and \texttt{samei@math.usask.ca}}

\bigskip
\leftline{Ross Stokke}
\leftline{Department of Mathematics and Statistics}
\leftline{University of Winnipeg}
\leftline{Winnipeg (MB), Canada R3B 2E9}
\leftline{\texttt{r.stokke@uwinnipeg.ca}}

\end{document}